\documentclass[12pt, a4paper]{amsart}

\usepackage[margin=2.2cm]{geometry}
\usepackage{marginnote}
\usepackage{float}
\usepackage[final]{graphicx}
\usepackage{epstopdf} 
\usepackage{caption}
\usepackage{subcaption}

\usepackage[T1]{fontenc}
\usepackage{mathrsfs}
\usepackage{mathtools}

\usepackage{tikz}
\usetikzlibrary{decorations.markings}
\usetikzlibrary{decorations.pathreplacing}
\usetikzlibrary{arrows,shapes,positioning,patterns}
\usetikzlibrary{knots}
\tikzstyle{none}=[inner sep=0pt]
\pgfdeclarelayer{edgelayer}
\pgfdeclarelayer{nodelayer}
\pgfsetlayers{edgelayer,nodelayer,main}
\usepackage{url}

\usepackage{verbatim}

\newcommand{\bR}{\mathbb R}

\newcommand{\bC}{\mathbb C}

\newcommand{\bZ}{\mathbb Z}

\newcommand{\Mat}{{\rm Mat}_{X}(\bC)}

\newcommand{\al}{\alpha}

\newcommand{\ep}{\varepsilon}
\newcommand{\si}{\sigma}
\newcommand{\om}{\omega}

\newcommand{\Ga}{\Gamma}

\newcommand{\ze}{\zeta}

\newcommand{\x}{\times}

\newcommand{\co}{\thinspace\colon}

\newcommand*{\leftcross}{%
	\raisebox{-5pt}{%
		\begin{tikzpicture}[scale=0.3, every path/.style ={very thick}, 
		every node/.style={knot crossing, inner sep = 2pt}]
		\node (l1) at (-1,-1) {};
		\node (l2) at (-1,1) {};
		\node (r1) at (1,-1) {};
		\node (r2) at (1,1) {};
		\node (m) at (0,0) {};
		\draw (l1.center) -- (m) {};
		\draw (m) -- (r2.center) {};
		\draw (l2.center) -- (r1.center);
		\end{tikzpicture}}
}

\newcommand*{\horres}{%
	\raisebox{-5pt}{%
		\begin{tikzpicture}[scale=0.3, every path/.style ={very thick}, 
		every node/.style={knot crossing, inner sep = 2pt}]
		\node (l1) at (-1,-1) {};
		\node (l2) at (-1,1) {};
		\node (r1) at (1,-1) {};
		\node (r2) at (1,1) {};
		\node (m) at (0,0) {};
		\draw (l2.center) .. controls (l2.2 south east) and (r2.2 south west) .. (r2.center) {};
		\draw (l1.center) .. controls (l1.2 north east) and (r1.2 north west) .. (r1.center) {};
		\end{tikzpicture}}
}

\newcommand*{\vertres}{%
	\raisebox{-5pt}{%
		\begin{tikzpicture}[scale=0.3, every path/.style ={very thick}, 
		every node/.style={knot crossing, inner sep = 2pt}]
		\node (l1) at (-1,-1) {};
		\node (l2) at (-1,1) {};
		\node (r1) at (1,-1) {};
		\node (r2) at (1,1) {};
		\node (m) at (0,0) {};
		\draw (l2.center) .. controls (l2.2 south east) and (l1.2 north east) .. (l1.center) {};
		\draw (r2.center) .. controls (r2.2 south west) and (r1.2 north west) .. (r1.center) {};
		\end{tikzpicture}}
}

\newcommand*{\leftcrossaxis}{%
	\raisebox{-5pt}{%
		\begin{tikzpicture}[scale=0.3, every path/.style ={very thick}, 
		every node/.style={knot crossing, inner sep = 2pt}]
		\node (l1) at (-1,-1) {};
		\node (l2) at (-1,1) {};
		\node (r1) at (1,-1) {};
		\node (r2) at (1,1) {};
		\node (m) at (0,0) {};
		\draw (l1.center) -- (m) {};
		\draw (m) -- (r2.center) {};
		\draw (l2.center) -- (r1.center);
		\draw[densely dashed, red] (0,-1) -- (0,1); 
		\end{tikzpicture}}
}

\newcommand*{\rightcrossaxis}{%
	\raisebox{-5pt}{%
		\begin{tikzpicture}[scale=0.3, every path/.style ={very thick}, 
		every node/.style={knot crossing, inner sep = 2pt}]
		\node (l1) at (-1,-1) {};
		\node (l2) at (-1,1) {};
		\node (r1) at (1,-1) {};
		\node (r2) at (1,1) {};
		\node (m) at (0,0) {};
		\draw (l1.center) -- (r2.center) {};
		\draw (l2.center) -- (m) {};
		\draw (m) -- (r1.center);
		\draw[densely dashed, red] (0,-1) -- (0,1){}; 
		\end{tikzpicture}}
}

\newcommand*{\horresaxis}{%
	\raisebox{-5pt}{%
		\begin{tikzpicture}[scale=0.3, every path/.style ={very thick}, 
		every node/.style={knot crossing, inner sep = 2pt}]
		\node (l1) at (-1,-1) {};
		\node (l2) at (-1,1) {};
		\node (r1) at (1,-1) {};
		\node (r2) at (1,1) {};
		\node (m) at (0,0) {};
		\draw (l2.center) .. controls (l2.2 south east) and (r2.2 south west) .. (r2.center) {};
		\draw (l1.center) .. controls (l1.2 north east) and (r1.2 north west) .. (r1.center) {};
		\draw[densely dashed, red] (0,-1) -- (0,1){}; 
		\end{tikzpicture}}
}

\newcommand*{\vertresaxis}{%
	\raisebox{-5pt}{%
		\begin{tikzpicture}[scale=0.3, every path/.style ={very thick}, 
		every node/.style={knot crossing, inner sep = 2pt}]
		\node (l1) at (-1,-1) {};
		\node (l2) at (-1,1) {};
		\node (r1) at (1,-1) {};
		\node (r2) at (1,1) {};
		\node (m) at (0,0) {};
		\draw (l2.center) .. controls (l2.2 south east) and (l1.2 north east) .. (l1.center) {};
		\draw (r2.center) .. controls (r2.2 south west) and (r1.2 north west) .. (r1.center) {};
		\draw[densely dashed, red] (0,-1) -- (0,1){}; 
		\end{tikzpicture}}
}

\newtheorem{thm}{Theorem}[section]
\newtheorem{lemma}[thm]{Lemma}
\newtheorem{cor}[thm]{Corollary}
\newtheorem{rem}[thm]{Remark}
\newtheorem{prop}[thm]{Proposition}

\newcommand*\sm[1]{\left(\begin{smallmatrix}#1\end{smallmatrix}\right)}
\newcommand*\mat[1]{\begin{pmatrix}#1\end{pmatrix}}

\theoremstyle{definition}

\newtheorem{defns}[thm]{Definitions}
\newtheorem{rmk}[thm]{Remark}

\newtheorem*{exas}{Examples}

\numberwithin{equation}{section}

\begin{document}

\title{Symmetric union diagrams and refined spin models}

\author{Carlo Collari}
\address{Mathematical Sciences, Durham University, UK}
\email{carlo.collari.math@gmail.com}
\author{Paolo Lisca}
\address{Department of Mathematics, University of Pisa, ITALY} 
\email{paolo.lisca@unipi.it}
\begin{abstract} 
An open question akin to the slice-ribbon conjecture asks whether every ribbon knot can be represented as a symmetric union. Next to this basic existence question sits the question of uniqueness of such representations. Eisermann and Lamm investigated the latter question by introducing a notion of symmetric equivalence among symmetric union diagrams and showing that inequivalent diagrams can be detected using a refined version of the Jones polynomial. We prove that every topological spin model gives rise to many effective invariants of symmetric equivalence, which can be used to distinguish infinitely many symmetric union diagrams representing the same link. We also show that such invariants are distinct from the refined Jones polynomial and we use them to provide a partial answer to a question left open by Eisermann and Lamm.
\end{abstract}

\maketitle

\section{Introduction}\label{s:intro}

\subsection{Symmetric diagrams and symmetric equivalences}\label{ss:sdse}

Let $\rho\co\bR^2\to\bR^2$ be the reflection given by $\rho(x,y)=(-x,y)$. The map $\rho$ fixes 
pointwise the subset $B =\{0\}\x\bR\subset\bR^2$, which will be called the \emph{axis}. 
Two diagrams $D,D'\subset\bR^2$ will be considered identical if there is an orientation-preserving 
diffeomorphism $h\co\bR^2\to\bR^2$ such that $h\circ\rho=\rho\circ h$ and $h(D)=D'$. 

An oriented link diagram $D\subset\bR^2$ is~\emph{symmetric} if $\rho(D) = \bar{D}$, where
$\bar{D}$ is the oriented diagram obtained from $D$ by reversing the orientation and 
switching all the crossings on the axis. 
A symmetric diagram $D$ is a~\emph{symmetric union} if $\rho$ sends each component $c_D$ of $D$ 
to itself in an orientation-reversing fashion, implying that $c_D$ crosses the 
axis perpendicularly in exactly two non--crossing points. Figure~\ref{f:89amphi} shows 
two unoriented symmetric union diagrams of the amphicheiral knot $8_9$. 
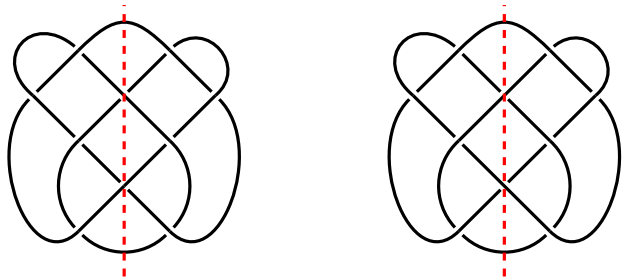
\begin{figure}[ht]
	\centering
	\begin{tikzpicture}[scale = 0.2, every path/.style ={very thick}] 
	\begin{pgfonlayer}{nodelayer}
		\node [style=none] (0) at (3, 0) {};
		\node [style=none] (1) at (0, -3) {};
		\node [style=none] (2) at (-3, 0) {};
		\node [style=none] (3) at (3, -6) {};
		\node [style=none] (4) at (-3, -6) {};
		\node [style=none] (5) at (0, 3) {};
		\node [style=none] (6) at (3, 6) {};
		\node [style=none] (7) at (-2.75, 5.75) {};
		\node [style=none] (8) at (6, 3) {};
		\node [style=none] (9) at (-6, 3) {};
		\node [style=none] (10) at (-3, 6) {};
		\node [style=none] (11) at (-0.25, 2.75) {};
		\node [style=none] (12) at (0.25, 3.25) {};
		\node [style=none] (13) at (2.75, -0.25) {};
		\node [style=none] (14) at (3.25, 0.25) {};
		\node [style=none] (15) at (0.25, -3.25) {};
		\node [style=none] (16) at (-0.25, -2.75) {};
		\node [style=none] (17) at (-2.75, -0.25) {};
		\node [style=none] (18) at (-3.25, 0.25) {};
		\node [style=none] (19) at (-3.25, 6.25) {};
		\node [style=none] (20) at (2.75, 5.75) {};
		\node [style=none] (21) at (3.25, 6.25) {};
		\node [style=none] (22) at (5.75, 3.25) {};
		\node [style=none] (23) at (6.25, 2.75) {};
		\node [style=none] (24) at (2.75, -6.25) {};
		\node [style=none] (25) at (3.25, -5.75) {};
		\node [style=none] (26) at (-2.75, -6.25) {};
		\node [style=none] (27) at (-3.25, -5.75) {};
		\node [style=none] (28) at (-6.25, 2.75) {};
		\node [style=none] (29) at (-5.75, 3.25) {};
	\end{pgfonlayer}
	\begin{pgfonlayer}{edgelayer}
		\draw (7.center) to (0.center);
		\draw [bend left=90, looseness=2.25] (9.center) to (19.center);
		\draw (29.center) to (10.center);
		\draw [bend left=45, looseness=1.50] (10.center) to (6.center);
		\draw (12.center) to (20.center);
		\draw (11.center) to (2.center);
		\draw (18.center) to (9.center);
		\draw (17.center) to (16.center);
		\draw (13.center) to (1.center);
		\draw (14.center) to (8.center);
		\draw [bend right=90, looseness=1.75] (8.center) to (21.center);
		\draw (6.center) to (22.center);
		\draw [bend left=45, looseness=1.00] (0.center) to (25.center);
		\draw (15.center) to (3.center);
		\draw [in=-45, out=-45, looseness=1.25] (3.center) to (23.center);
		\draw [bend right=45, looseness=1.00] (2.center) to (27.center);
		\draw [bend right=45, looseness=1.00] (26.center) to (24.center);
		\draw (1.center) to (4.center);
		\draw [in=-135, out=-135, looseness=1.25] (28.center) to (4.center);
		\draw[dashed, red] (0,-9) -- (0,9); 
	\end{pgfonlayer}
\begin{scope}[xshift = 25cm]
\begin{pgfonlayer}{nodelayer}
\node [style=none] (0) at (3, 0) {};
\node [style=none] (1) at (-0.25, -3.25) {};
\node [style=none] (2) at (-3, 0) {};
\node [style=none] (3) at (3, -6) {};
\node [style=none] (4) at (-3, -6) {};
\node [style=none] (5) at (3, 6) {};
\node [style=none] (6) at (-2.75, 5.75) {};
\node [style=none] (7) at (6, 3) {};
\node [style=none] (8) at (-6, 3) {};
\node [style=none] (9) at (-3, 6) {};
\node [style=none] (10) at (2.75, 5.75) {};
\node [style=none] (11) at (2.75, -0.25) {};
\node [style=none] (12) at (3.25, 0.25) {};
\node [style=none] (13) at (-2.75, -0.25) {};
\node [style=none] (14) at (-3.25, 0.25) {};
\node [style=none] (15) at (-3.25, 6.25) {};
\node [style=none] (16) at (2.75, 5.75) {};
\node [style=none] (17) at (3.25, 6.25) {};
\node [style=none] (18) at (5.75, 3.25) {};
\node [style=none] (19) at (6.25, 2.75) {};
\node [style=none] (20) at (2.75, -6.25) {};
\node [style=none] (21) at (3.25, -5.75) {};
\node [style=none] (22) at (-2.75, -6.25) {};
\node [style=none] (23) at (-3.25, -5.75) {};
\node [style=none] (24) at (-6.25, 2.75) {};
\node [style=none] (25) at (-5.75, 3.25) {};
\node [style=none] (26) at (0.25, 2.75) {};
\node [style=none] (27) at (-0.25, 3.25) {};
\node [style=none] (28) at (0.25, -2.75) {};
\end{pgfonlayer}
\begin{pgfonlayer}{edgelayer}
\draw [bend left=90, looseness=2.25] (8.center) to (15.center);
\draw (25.center) to (9.center);
\draw [bend left=45, looseness=1.50] (9.center) to (5.center);
\draw (10.center) to (2.center);
\draw (14.center) to (8.center);
\draw (12.center) to (7.center);
\draw [bend right=90, looseness=1.75] (7.center) to (17.center);
\draw (5.center) to (18.center);
\draw [bend left=45, looseness=1.00] (0.center) to (21.center);
\draw (13.center) to (3.center);
\draw [in=-45, out=-45, looseness=1.25] (3.center) to (19.center);
\draw [bend right=45, looseness=1.00] (2.center) to (23.center);
\draw [bend right=45, looseness=1.00] (22.center) to (20.center);
\draw (1.center) to (4.center);
\draw [in=-135, out=-135, looseness=1.25] (24.center) to (4.center);
\draw (6.center) to (27.center);
\draw (26.center) to (0.center);
\draw (28.center) to (11.center);
\draw[dashed, red] (0,-9) -- (0,9); 
\end{pgfonlayer}
\end{scope}
\end{tikzpicture}
	\caption{Symmetric union diagrams of the knot $8_9$}
	\label{f:89amphi}
\end{figure}
The two diagrams are obtained from each other by switching all the crossings on the axis, which 
amounts to reflecting across the plane of the page and then applying a 3-dimensional $180^0$ rotation around the axis. Eisermann and Lamm~\cite[\S 2.4]{Ei.La11} observe that to each symmetric diagram one can associate 
a singular link $L\subset\bR^3$ with some extra data.  
This is done by converting each crossing on the axis into a double point belonging to the plane 
$E = \{x=0\}\subset\bR^3$, and encoding the over-under crossing information by a sign 
attached to the double point according to the rules of Figure~\ref{f:double-points}. 
\begin{figure}[ht]
	\centering	
		\begin{tikzpicture}[scale=0.4, every path/.style ={very thick}, 
	every node/.style={knot crossing, inner sep=2pt}] 
	\node (l-1) at (-1,-1){};
	\node (l1) at (-1,1){};
	\node (m) at (0,0){};
	\node (r-1) at (1,-1){};
	\node (r1) at (1,1){};
	\draw (l1.center) -- (r-1.center){};
	\draw (l-1.center) -- (m){};
	\draw (m) -- (r1.center){};
	\draw[dashed, red] (0,-1.5) -- (0,1.5);
	\begin{scope}[xshift=3cm]
	\draw[->] (-1,0) -- (1,0){};
	\end{scope}
	\begin{scope}[xshift=6cm]
	\node (l-1) at (-1,-1){};
	\node (l1) at (-1,1){};
	\node at (1,0){$-$};
	\node (r-1) at (1,-1){};
	\node (r1) at (1,1){};
	\draw (l1.center) -- (r-1.center){};
	\draw (l-1.center) -- (r1.center){};
	\filldraw[fill=black] (0,0) circle (5pt){};	
	\draw[dashed, red] (0,-1.5) -- (0,1.5);
	\end{scope}
	\begin{scope}[xshift=12cm]
	\node (l-1) at (-1,-1){};
	\node (l1) at (-1,1){};
	\node (m) at (0,0){};
	\node (r-1) at (1,-1){};
	\node (r1) at (1,1){};
	\draw (l-1.center) -- (r1.center){};
	\draw (l1.center) -- (m){};
	\draw (m) -- (r-1.center){};
	\draw[dashed, red] (0,-1.5) -- (0,1.5);
	\end{scope}
	\begin{scope}[xshift=15cm]
	\draw[->] (-1,0) -- (1,0){};
	\end{scope}
	\begin{scope}[xshift=18cm]
	\node (l-1) at (-1,-1){};
	\node (l1) at (-1,1){};
	\node at (1,0){$+$};
	\node (r-1) at (1,-1){};
	\node (r1) at (1,1){};
	\draw (l1.center) -- (r-1.center){};
	\draw (l-1.center) -- (r1.center){};
	\filldraw[fill=black] (0,0) circle (5pt){};	
	\draw[dashed, red] (0,-1.5) -- (0,1.5);
	\end{scope}
	\end{tikzpicture}
	\caption{How to turn a crossing on the axis into a signed double point}
	\label{f:double-points}
\end{figure}
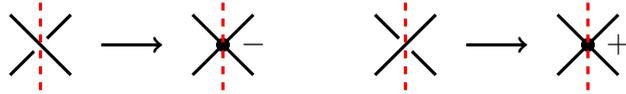  
The resulting singular link with signs, transverse to $E$ and invariant under reflection 
with respect to $E$, is what we call a~\emph{symmetric singular link}. We say that two 
symmetric diagrams are \emph{strongly symmetrically equivalent} if their associated  
symmetric singular links can be connected via a smooth family of symmetric singular links. 
Eisermann and Lamm~\cite[Theorem~2.12]{Ei.La11} show that symmetric diagrams satisfy 
a symmetric version of the Reidemeister theorem, where the symmetric analogues of 
the Reidemeister moves relating two symmetric diagrams are defined as follows.

A~\emph{symmetric Reidemeister move off the axis} is an ordinary Reidemeister move 
carried out, away from the axis $B$, together with its mirror-symmetric counterpart with respect to 
$B$. A~\emph{symmetric Reidemeister move on the axis} is one of the moves 
$S2(h)$, $S2(\pm)$, $S3(o\pm)$, $S3(u\pm)$ and 
$S4(\pm\pm)$, some of which are illustrated in Figure~\ref{f:sRm} (see~\cite[\S 2.3]{Ei.La11}
for the complete list).
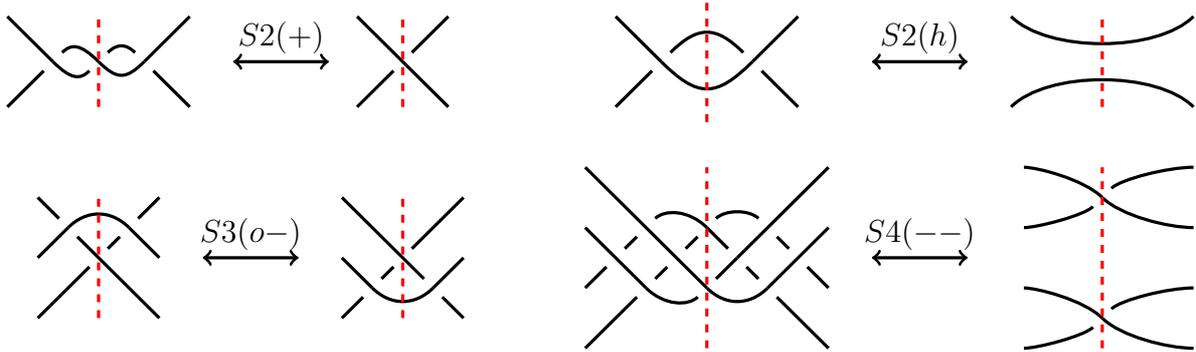
\begin{figure}[ht]
	\centering
	\begin{tikzpicture}[scale=0.4, every path/.style ={very thick}, 
every node/.style={knot crossing, inner sep = 3pt}] 
\node (l1) at (-2,-2) {};
\node (l2) at (-2,0) {};
\node (l3) at (-1,1) {};
\node (l4) at (-2,2) {};
\node (r1) at (2,-2) {};
\node (r2) at (2,0) {};
\node (r3) at (1,1) {};
\node (r4) at (2,2) {};
\node (m) at (0,0) {};
\draw (l1.center) -- (m) {};
\draw (m) -- (r3) {};
\draw (l3) -- (r1.center);
\draw (l4.center) -- (l3) {};
\draw (r3) -- (r4.center) {};
\draw (l2.center) .. controls (l2.2 north east) and (l3.2 south west) .. (l3.center);
\draw (l3.center) .. controls (l3.2 north east) and (r3.2 north west) .. (r3.center);
\draw (r3.center) .. controls (r3.2 south east) and (r2.2 north west) .. (r2.center);
\draw[dashed, red] (0,-2) -- (0,2); 
\begin{scope}[xshift=5cm]
\node (l) at (-2,0) {};
\node (r) at (2,0) {};
\node at (0,0.8) {$S3(o-)$};
\draw[<->] (l) -- (r);
\end{scope}
\begin{scope}[xshift=10cm]
\node (l1) at (-2,-2) {};
\node (l2) at (-1,-1) {};
\node (l3) at (-2,0) {};
\node (l4) at (-2,2) {};
\node (r1) at (2,-2) {};
\node (r2) at (1,-1) {};
\node (r3) at (2,0) {};
\node (r4) at (2,2) {};
\node (m) at (0,0) {};
\draw (l1.center) -- (l2) {};
\draw (m) -- (r4.center) {};
\draw (r2) -- (r1.center);
\draw (l2) -- (m) {};
\draw (l4.center) -- (r2) {};
\draw (l3.center) .. controls (l3.2 south east) and (l2.2 north west) .. (l2.center);
\draw (l2.center) .. controls (l2.2 south east) and (r2.2 south west) .. (r2.center);
\draw (r2.center) .. controls (r2.2 north east) and (r3.2 south west) .. (r3.center);
\draw[dashed, red] (0,-2) -- (0,2); 	
\end{scope}
\begin{scope}[yshift=5cm]
\node (l1) at (-3,0) {};
\node (l2) at (-1.5,1.5) {};
\node (l3) at (-3,3) {};
\node (m) at (0,1.5) {};
\node (r1) at (3,0) {};
\node (r2) at (1.5,1.5) {};
\node (r3) at (3,3) {};
\draw (l1.center) .. controls (l1.2 north east) and (l2.2 south west) .. (l2);
\draw (l2.center) .. controls (l2.2 north west) and (l3.2 south east) .. (l3.center);
\draw (r1.center) .. controls (r1.2 north west) and (r2.2 south east) .. (r2);
\draw (r2.center) .. controls (r2.2 north east) and (r3.2 south west) .. (r3.center);
\draw (l2.center) .. controls (l2.2 south east) and (m.2 south west) .. (m);
\draw (l2) .. controls (l2.2 north east) and (m.2 north west) .. (m.center);
\draw (r2.center) .. controls (r2.2 south west) and (m.2 south east) .. (m.center);
\draw (r2) .. controls (r2.2 north west) and (m.2 north east) .. (m);
\draw[dashed, red] (0,0) -- (0,3); 
\end{scope}
\begin{scope}[xshift=6cm,yshift=5cm]
\node (l) at (-2,1.5) {};
\node (r) at (2,1.5) {};
\node at (0,2.3) {$S2(+)$};
\draw[<->] (l) -- (r);
\end{scope}
\begin{scope}[xshift=10cm,yshift=5cm]
\node (l1) at (-1.5,0) {};
\node (l3) at (-1.5,3) {};
\node (m) at (0,1.5) {};
\node (r1) at (1.5,0) {};
\node (r3) at (1.5,3) {};
\draw (l1.center) -- (m);
\draw (l3.center) -- (r1.center);
\draw (m) -- (r3.center);
\draw[dashed, red] (0,0) -- (0,3); 
\end{scope}
\begin{scope}[xshift=20cm]
\node (l4-3) at (-4,-3) {};
\node (l4-1) at (-4,-1) {};
\node (l41) at (-4,1) {};
\node (l43) at (-4,3) {};
\node (l30) at (-3,0) {};
\node (l2-1) at (-2,-1) {};
\node (l21) at (-2,1) {};
\node (l10) at (-1,0) {};
\node (m-2) at (0,-1) {};
\node (m2) at (0,1) {};
\node (r4-3) at (4,-3) {};
\node (r4-1) at (4,-1) {};
\node (r41) at (4,1) {};
\node (r43) at (4,3) {};
\node (r30) at (3,0) {};
\node (r2-1) at (2,-1) {};
\node (r21) at (2,1) {};
\node (r10) at (1,0) {};
\draw (l4-3.center) -- (l2-1){};
\draw (l4-1.center) -- (l30){};
\draw (l41.center) -- (l2-1.center){};
\draw (l43.center) -- (l10.center){};
\draw (l30) -- (l21){};
\draw (l2-1) -- (l10){};
\draw (l2-1.center) .. controls (l2-1.2 south east) and (m-2.2 south west) .. (m-2){};
\draw (l10.center) .. controls (l10.2 south east) and (m-2.2 north west) .. (m-2.center){};
\draw (l21) .. controls (l21.2 north east) and (m2.2 north west) .. (m2.center){};
\draw (l10) .. controls (l10.2 north east) and (m2.2 south west) .. (m2){};
\draw (m2.center) .. controls (m2.2 south east) and (r10.2 north west) .. (r10){};
\draw (m2) .. controls (m2.2 north east) and (r21.2 north west) .. (r21){};
\draw (m-2.center) .. controls (m-2.2 south east) and (r2-1.2 south west) .. (r2-1.center){};
\draw (m-2) .. controls (m-2.2 north east) and (r10.2 south west) .. (r10.center) {};
\draw (r10.center) -- (r43.center){};
\draw (r10) -- (r2-1){};
\draw (r2-1.center) -- (r41.center){};
\draw (r21) -- (r30){};
\draw (r30) -- (r4-1.center){};
\draw (r2-1) -- (r4-3.center){};
\draw[dashed, red] (0,-3) -- (0,3){};
\end{scope}
\begin{scope}[xshift=27cm]
\node (l) at (-2,0) {};
\node (r) at (2,0) {};
\node at (0,0.8) {$S4(--)$};
\draw[<->] (l) -- (r);
\end{scope}
\begin{scope}[xshift=33cm]
\node (l3-3) at (-3,-3){};
\node (l3-1) at (-3,-1){};
\node (l31) at (-3,1){};
\node (l33) at (-3,3){};
\node (m-2) at (0,-2){};
\node (m2) at (0,2){};
\node (r3-3) at (3,-3){};
\node (r3-1) at (3,-1){};
\node (r31) at (3,1){};
\node (r33) at (3,3){};
\draw (l33) .. controls (l33.2 east) and (m2.2 north west) .. (m2.center){}; 
\draw (l31) .. controls (l31.2 east) and (m2.2 south west) .. (m2){}; 
\draw (l3-1) .. controls (l3-1.2 east) and (m-2.2 north west) .. (m-2.center){}; 
\draw (l3-3) .. controls (l3-3.2 east) and (m-2.2 south west) .. (m-2){}; 
\draw (m-2.center) .. controls (m-2.2 south east) and (r3-3.2 west) .. (r3-3.center){};
\draw (m-2) .. controls (m-2.2 north east) and (r3-1.2 west) .. (r3-1.center){};
\draw (m2.center) .. controls (m2.2 south east) and (r31.2 west) .. (r31.center){};
\draw (m2) .. controls (m2.2 north east) and (r33.2 west) .. (r33.center){};
\draw[dashed, red] (0,-3) -- (0,3);
\end{scope}
\begin{scope}[xshift=20cm,yshift=5cm]
\node (l1) at (-3,0) {};
\node (l2) at (-1.5,1.5) {};
\node (l3) at (-3,3) {};
\node (r1) at (3,0) {};
\node (r2) at (1.5,1.5) {};
\node (r3) at (3,3) {};
\draw (l1.center) .. controls (l1.2 north east) and (l2.2 south west) .. (l2);
\draw (l2.center) .. controls (l2.2 north west) and (l3.2 south east) .. (l3.center);
\draw (r1.center) .. controls (r1.2 north west) and (r2.2 south east) .. (r2);
\draw (r2.center) .. controls (r2.2 north east) and (r3.2 south west) .. (r3.center);
\draw (l2.center) .. controls (l2.4 south east) and (r2.4 south west) .. (r2.center);
\draw (l2) .. controls (l2.4 north east) and (r2.4 north west) .. (r2);
\draw[dashed, red] (0,-0.5) -- (0,3.5); 
\end{scope}
\begin{scope}[xshift=27cm,yshift=5cm]
\node (l) at (-2,1.5) {};
\node (r) at (2,1.5) {};
\node at (0,2.3) {$S2(h)$};
\draw[<->] (l) -- (r);
\end{scope}
\begin{scope}[xshift=33cm,yshift=5cm]
\node (l1) at (-3,0) {};
\node (l3) at (-3,3) {};
\node (r1) at (3,0) {};
\node (r3) at (3,3) {};
\draw (l1.center) .. controls (l1.4 north east) and (r1.4 north west) .. (r1.center);
\draw (l3.center) .. controls (l3.4 south east) and (r3.4 south west) .. (r3.center);
\draw[dashed, red] (0,0) -- (0,3); 
\end{scope}
\end{tikzpicture}
	\caption{Representative symmetric Reidemeister moves}
	\label{f:sRm}
\end{figure} 
It is understood that these moves admit variants obtained by 
turning the corresponding pictures upside down, mirroring or rotating them around the axis. 
With our present terminology, Eisermann and Lamm prove the following. 
\begin{thm}[Symmetric Reidemeister Theorem {{\cite[Theorem~2.12]{Ei.La11}}}]\label{t:sRm}
	Two symmetric diagrams are strongly symmetrically equivalent if and only if they 
	can be obtained from each other via a finite sequence of symmetric Reidemeister moves. 
	\footnote{Eisermann and Lamm state only one of the two implications of Theorem~\ref{t:sRm}, but they 
		use a terminology slightly different from ours and they include the moves $S1$ and $S2(v)$
		of Figure~\ref{f:extrasRm} among their symmetric Reidemeister moves. It can be easily checked that in our terminology and 
		for the set of moves of Figure~\ref{f:sRm}, Theorem~2.12 from~\cite{Ei.La11} is equivalent to Theorem~\ref{t:sRm} (cf.~\cite[Remark~2.13]{Ei.La11}).} 
\end{thm}
Eisermann and Lamm show~\cite[Example~6.8]{Ei.La11} that the two symmetric union 
diagrams of the knot $8_9$ given in Figure~\ref{f:89amphi} are strongly symmetrically equivalent.
On the other hand, they also consider another symmetric union diagram for the knot $8_9$,   
i.e.~the left-most diagram of Figure~\ref{f:89}, as well as the center and right-most diagrams in Figure~\ref{f:89}, which are two symmetric union diagrams for the knot $10_{42}$. 
\begin{figure}[ht]
	\centering
		\begin{tikzpicture}[scale = 0.5, every path/.style ={very thick}, 
	every node/.style={knot crossing, inner sep = 2pt}] 
	\node (l1) at (-1,0) {};
	\node (l2) at (-1,2) {};
	\node (l3) at (-1,3) {};
	\node (l4) at (-1,5) {};
	\node (l5) at (-1,7) {};
	\node (m1) at (0,1) {};
	\node (m2) at (0,4) {};
	\node (m3) at (0,6) {};
	\node (r1) at (1,0) {};
	\node (r2) at (1,2) {};
	\node (r3) at (1,3) {};
	\node (r4) at (1,5) {};
	\node (r5) at (1,7) {};
	\draw (l1) -- (m1);
	\draw (r1) -- (m1.center);
	\draw (m1.center) -- (l2.center);
	\draw (m1) -- (r2.center);
	\draw (l1.center) .. controls (l1.2 south east) and (r1.2 south west) .. (r1.center);
	\draw (l1.center) .. controls (l1.4 north west) and (l2.4 south west) .. (l2);
	\draw (r1.center) .. controls (r1.4 north east) and (r2.4 south east) .. (r2);
	\draw (l2.center) .. controls (l2.2 north west) and (l3.2 south west) .. (l3);
	\draw (r2.center) .. controls (r2.2 north east) and (r3.2 south east) .. (r3);
	\draw (l2) .. controls (l2.2 north east) and (l3.2 south east) .. (l3.center);
	\draw (r2) .. controls (r2.2 north west) and (r3.2 south west) .. (r3.center);
	\draw (l3.center) .. controls (l3.4 north west) and (l4.4 south west) .. (l4);
	\draw (r3.center) .. controls (r3.4 north east) and (r4.4 south east) .. (r4);
	\draw (l3) -- (r4.center);
	\draw (r3) -- (m2);
	\draw (m2) -- (l4.center);
	\draw (l4.center) .. controls (l4.4 north west) and (l5.4 south west) .. (l5);
	\draw (r4.center) .. controls (r4.4 north east) and (r5.4 south east) .. (r5);
	\draw (l4) -- (m3);
	\draw (r4) -- (l5.center);
	\draw (m3) -- (r5.center);
	\draw (l5) .. controls (l5.2 north east) and (r5.2 north west) .. (r5);
	\draw (l1) .. controls (l1.16 south west) and (l5.16 north west) .. (l5.center);
	\draw (r1) .. controls (r1.16 south east) and (r5.16 north east) .. (r5.center);
	\draw[dashed, red] (0,-1) -- (0,8); 
\begin{scope}[xshift=10cm, yshift=2cm]
\node (l1) at (-1.5,-2.3){};
\node (l2) at (-1,-1){};
\node (l3) at (-1,1){};
\node (l4) at (-1,3){};
\node (l5) at (-2,4){};
\node (l6) at (-1,5){};
\node (r1) at (1.5,-2.3){};
\node (r2) at (1,-1){};
\node (r3) at (1,1){};
\node (r4) at (1,3){};
\node (r5) at (2,4){};
\node (r6) at (1,5){};
\node (m1) at (0,-2){};
\node (m2) at (0,0){};
\node (m3) at (0,2){};
\node (m4) at (0,4){};
\draw (l1.center) .. controls (l1.2 east) and (m1.2 south west) .. (m1){};
\draw (l1.center) .. controls (l1.8 west) and (l5.8 south west) .. (l5){};
\draw (l1) .. controls (l1.4 south) and (r1.4 south) .. (r1){};
\draw (l1) .. controls (l1.2 north) and (l2.2 south west) .. (l2.center){};
\draw (r1.center) .. controls (r1.2 west) and (m1.2 south east) .. (m1.center){};
\draw (r1.center) .. controls (r1.8 east) and (r5.8 south east) .. (r5){};
\draw (r1) .. controls (r1.2 north) and (r2.2 south east) .. (r2.center){};
\draw (m1.center) .. controls (m1.2 north west) and (l2.2 south east) .. (l2){};
\draw (m1) .. controls (m1.2 north east) and (r2.2 south west) .. (r2){};
\draw (l2) .. controls (l2.4 north west) and (l3.4 south west) .. (l3.center){};
\draw (l2.center) .. controls (l2.2 north east) and (m2.2 south west) .. (m2.center){};
\draw (m2) .. controls (m2.2 north west) and (l3.2 south east) .. (l3){};
\draw (m2.center) .. controls (m2.2 north east) and (r3.2 south west) .. (r3){};
\draw (r2.center) .. controls (r2.2 north west) and (m2.2 south east) .. (m2){};
\draw (r2) .. controls (r2.4 north east) and (r3.4 south east) .. (r3.center){};
\draw (l3.center) .. controls (l3.2 north east) and (m3.2 south west) .. (m3){};
\draw (r3.center) .. controls (r3.2 north west) and (m3.2 south east) .. (m3.center){};
\draw (l3) .. controls (l3.4 north west) and (l4.4 south west) .. (l4.center){};
\draw (r3) .. controls (r3.4 north east) and (r4.4 south east) .. (r4.center){};
\draw (l4) .. controls (l4.2 north west) and (l5.2 south east) .. (l5.center){};
\draw (r4) .. controls (r4.2 north east) and (r5.2 south west) .. (r5.center){};
\draw (l4.center) .. controls (l4.2 north east) and (m4.2 south west) .. (m4.center){};
\draw (r4.center) .. controls (r4.2 north west) and (m4.2 south east) .. (m4){};
\draw (l4) .. controls (l4.2 south east) and (m3.2 north west) .. (m3.center){};
\draw (r4) .. controls (r4.2 south west) and (m3.2 north east) .. (m3){};
\draw (l5) .. controls (l5.2 north east) and (l6.2 south west) .. (l6.center){};
\draw (r5) .. controls (r5.2 north west) and (r6.2 south east) .. (r6.center){};
\draw (l5.center) .. controls (l5.4 north west) and (l6.4 north west) .. (l6){};
\draw (r5.center) .. controls (r5.4 north east) and (r6.4 north east) .. (r6){};
\draw (m4) .. controls (m4.2 north west) and (l6.2 south east) .. (l6){};
\draw (m4.center) .. controls (m4.2 north east) and (r6.2 south west) .. (r6){};
\draw (l6.center) .. controls (l6.4 north east) and (r6.4 north west) .. (r6.center){};
\draw [dashed, red] (0,-3.5) -- (0,6){};
\end{scope}
\begin{scope}[xshift=19cm, yshift=3.5cm]
\node (l1) at (-1,-3.2){};
\node (l2) at (-1,-1.7){};
\node (l3) at (-1,0){};
\node (l4) at (-1,1.5){};
\node (l5) at (-2,2){};
\node (l6) at (-1,3){};
\node (m1) at (0,-1){};
\node (m2) at (0,1){};
\node (r1) at (1,-3.2){};
\node (r2) at (1,-1.7){};
\node (r4) at (1,1.5){};
\node (r3) at (1,0){};
\node (r5) at (2,2){};
\node (r6) at (1,3){};
\draw (l1) .. controls (l1.4 south east) and (r1.4 south west) .. (r1){};
\draw (l1.center) .. controls (l1.2 north east) and (l2.2 south east) .. (l2){};
\draw (r1.center) .. controls (r1.2 north west) and (r2.2 south west) .. (r2){};
\draw (l1) .. controls (l1.2 north west) and (l2.2 south west) ..  (l2.center){};
\draw (r1) .. controls (r1.2 north east) and (r2.2 south east) ..(r2.center){};
\draw (l2.center) .. controls (l2.2 north east) and (m1.2 south west) .. (m1.center){};
\draw (r2.center) .. controls (r2.2 north west) and (m1.2 south east) .. (m1){};
\draw (l2) .. controls (l2.2 north west) and (l3.2 south west) ..  (l3.center){};
\draw (r2) .. controls (r2.2 north east) and (r3.2 south east) ..(r3.center){};
\draw (l3) .. controls (l3.2 south east) and (m1.2 north west) .. (m1){};
\draw (r3) .. controls (r3.2 south west) and (m1.2 north east) .. (m1.center){};
\draw (l3) .. controls (l3.2 north west) and (l4.2 south west) .. (l4.center){};
\draw (r3) .. controls (r3.2 north east) and (r4.2 south east) .. (r4.center){};
\draw (l3.center) .. controls (l3.2 north east) and (m2.2 south west) .. (m2){};
\draw (r3.center) .. controls (r3.2 north west) and (m2.2 south east) .. (m2.center){};
\draw (l1.center) .. controls (l1.8 south west) and (l5.8 south west) .. (l5){};
\draw (r1.center) .. controls (r1.8 south east) and (r5.8 south east) .. (r5){};
\draw (l5.center) .. controls (l5.2 south east) and (l4.2 west) .. (l4){};
\draw (l4) .. controls (l4.2 east) and (m2.2 north west) .. (m2.center){};
\draw (r5.center) .. controls (r5.2 south west) and (r4.2 east) .. (r4){};
\draw (r4) .. controls (r4.2 west) and (m2.2 north east) .. (m2){};
\draw (l5) .. controls (l5.2 north east) and (l6.2 south west) .. (l6.center){};
\draw (r5) .. controls (r5.2 north west) and (r6.2 south east) .. (r6.center){};
\draw (l5.center) .. controls (l5.4 north west) and (l6.4 north west) .. (l6){};
\draw (r5.center) .. controls (r5.4 north east) and (r6.4 north east) .. (r6){};
\draw (l4.center) .. controls (l4.2 north east) and (l6.2 south east) .. (l6){};
\draw (r4.center) .. controls (r4.2 north west) and (r6.2 south west) .. (r6){};
\draw (l6.center) .. controls (l6.4 north east) and (r6.4 north west) .. (r6.center){};
\draw [dashed, red] (0,-4.3) -- (0,4){};
\end{scope}
	\end{tikzpicture}
	\caption{Symmetric union diagrams of $8_9$ (left) and $10_{42}$ (center and right)}
	\label{f:89}
\end{figure}
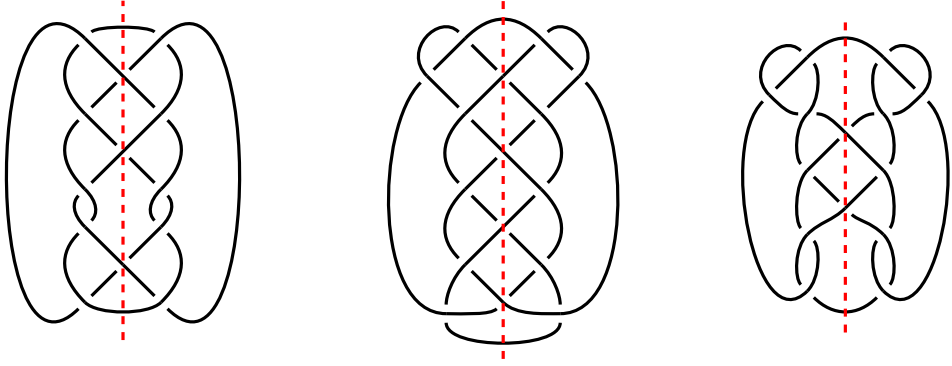
The two symmetric union diagrams of $10_{42}$ in Figure~\ref{f:89} are not 
strongly symmetrically equivalent because the associated symmetric singular links have 
different numbers of double points (four and two). Similarly, the  
symmetric union diagram of $8_9$ from Figure~\ref{f:89} is not strongly symmetrically equivalent
to the diagram of the same knot obtained by switching all the crossings on the axis, because 
the two diagrams have different numbers of \emph{signed} crossings on the axis. 
(Note that both diagrams represent $8_9$ because they clearly represent mirror 
equivalent knots, and $8_9$ is amphicheiral).

These examples show that the notion of strong symmetric equivalence is not a very subtle one, 
but Eisermann and Lamm consider two extra moves on symmetric diagrams, which they call   
$S1(\pm)$ and $S2(v)$. Some examples of the extra moves are illustrated in Figure~\ref{f:extrasRm}. 
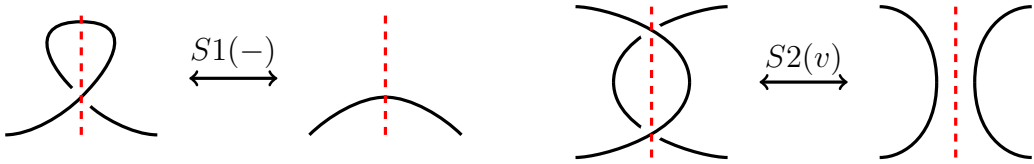
\begin{figure}[ht]
	\centering
	\pgfdeclarelayer{layer1}
\pgfdeclarelayer{layer2}
\pgfdeclarelayer{layer3}
\pgfsetlayers{layer1,layer2,layer3}
\begin{tikzpicture}[scale=0.5, every path/.style ={very thick}, 
every node/.style={knot crossing, inner sep = 3pt}] 
\begin{pgfonlayer}{layer1}
\begin{scope}[yshift=0.6cm]
\node (l1) at (-2,0) {};
\node (r1) at (2,0) {};
\node (c1) at (0,1) {};
\node (c2) at (0,3) {};
\draw (l1.center) .. controls (l1.2 east) and (c1.2 south west) .. (c1.center);
\draw (c1.center) .. controls (c1.4 north east) and (c2.4 east) .. (c2.center);
\draw (c2.center) .. controls (c2.4 west) and (c1.4 north west) .. (c1);
\draw (c1) .. controls (c1.2 south east) and (r1.2 west) .. (r1.center);
\draw[dashed, red] (0,0) -- (0,3.3); 
\end{scope}
\begin{scope}[xshift=4cm,yshift=0.6cm]
\node (l) at (-1.5,1.5) {};
\node (r) at (1.5,1.5) {};
\node at (0,2.2) {$S1(-)$};
\draw[<->] (l) -- (r);
\end{scope}
\begin{scope}[xshift=8cm,yshift=0.6cm]
\node (l1) at (-2,0) {};
\node (r1) at (2,0) {};
\node (c1) at (0,1) {};
\draw (l1.center) .. controls (l1.2 north east) and (c1.2 west) .. (c1.center);
\draw (c1.center) .. controls (c1.2 east) and (r1.2 north west) .. (r1.center);
\draw[dashed, red] (0,0) -- (0,3.3); 
\end{scope}
\begin{scope}[xshift=15cm]
\node (l1) at (-2,0) {};
\node (l2) at (-1,2) {};
\node (l3) at (-2,4) {};
\node (r1) at (2,0) {};
\node (r2) at (1,2) {};
\node (r3) at (2,4) {};
\begin{pgfonlayer}{layer2}
\node[circle,color=white,fill=white,radius=5pt] (c1) at (0,0.7) {};
\node[circle,color=white,fill=white,radius=5pt] (c2) at (0,3.3) {};
\end{pgfonlayer}
\begin{pgfonlayer}{layer3}
\draw (l1.center) .. controls (l1.2 east) and (r2.4 south) .. (r2.center);
\draw (r2.center) .. controls (r2.4 north) and (l3.2 east) .. (l3.center);
\draw[dashed, red] (0,0) -- (0,4); 
\end{pgfonlayer}
\draw (r1.center) .. controls (r1.2 west) and (l2.4 south) .. (l2.center);
\draw (l2.center) .. controls (l2.4 north) and (r3.2 west) .. (r3.center);
\end{scope}
\begin{scope}[xshift=19cm]
\node (l) at (-1.5,2) {};
\node (r) at (1.5,2) {};
\node at (0,2.6) {$S2(v)$};
\draw[<->] (l) -- (r);
\end{scope}
\begin{scope}[xshift=23cm]
\node (l1) at (-2,0) {};
\node (l2) at (-0.5,2) {};
\node (l3) at (-2,4) {};
\node (r1) at (2,0) {};
\node (r2) at (0.5,2) {};
\node (r3) at (2,4) {};
\draw (l1.center) .. controls (l1.2 east) and (l2.4 south) .. (l2.center);
\draw (l2.center) .. controls (l2.4 north) and (l3.2 east) .. (l3.center);
\draw (r1.center) .. controls (r1.2 west) and (r2.4 south) .. (r2.center);
\draw (r2.center) .. controls (r2.4 north) and (r3.2 west) .. (r3.center);
\draw[dashed, red] (0,0) -- (0,4); 
\end{scope}
\end{pgfonlayer}
\end{tikzpicture}
	\caption{Extra moves $S1(-)$ and $S2(v)$}
	\label{f:extrasRm}
\end{figure} 
\begin{defns}\label{d:SE}
Two oriented, symmetric diagrams which can be obtained from each other via a finite sequence 
of symmetric Reidemester (or $sR$) moves \emph{and $S1$ moves} will be 
called~\emph{symmetrically equivalent}. If they can be obtained from each other using $sR$ moves, 
$S1$ \emph{and $S2(v)$ moves}, we will say that the diagrams are~\emph{weakly symmetrically equivalent}. 
\end{defns}
The notions of symmetric equivalence introduced with Definitions~\ref{d:SE} are more subtle 
than strong symmetric equivalence: for instance, 
it is not obvious whether the two symmetric diagrams of $8_9$ and $10_{42}$ described above  
are symmetrically equivalent (weakly or not).

\subsection{Eisermann and Lamm's refined Jones polynomial and its applications}\label{ss:rJp}

To each oriented link diagram $D\subset\bR^2$ transverse 
to $B =\{0\}\x\bR$, Eisermann and Lamm associate an invariant of weak symmetric equivalence $W(D)$ taking values in the quotient field $\bZ(X_A,X_B)$ of the ring of Laurent polynomials in the variables $X_A$ and $X_B$ with integer coefficients. The invariant is defined by setting 
\[
W(D) = (-X_A^{-3})^{w_A(D)} (-X_B^{-3})^{w_B(D)} \langle D\rangle
\]
where $w_A(D)$ and $w_B(D)$ are, respectively, the sum of crossing signs off and on the axis, and 
$\langle D\rangle$ is a refined Kauffman bracket specified by the skein relation 
\[
\left<\leftcross{}\right> = X_A \left<\horres{} \right> + X_A^{-1} \left<\vertres\right>
\]
for crossings off the axis, the skein relations
\[
\left<\leftcrossaxis{}\right> = X_B \left<\horresaxis{} \right> + X_B^{-1} \left<\vertresaxis\right>
\quad\quad\quad
\left<\rightcrossaxis{}\right> = X_B^{-1} \left<\horresaxis{} \right> + X_B \left<\vertresaxis\right>
\]
for crossings on the axis, and taking the value 
\[
\left< C\right> = (-X_A^2-X_A^{-2}) ^{n-m} (-X_B^2-X_B^{-2})^{m-1}
\]
on a collection $C$ of $n$ circles intersecting the axis $B$ in $2m$ points. 

It turns out~\cite[Propostion~1.8]{Ei.La11} that when $D$ is a symmetric union knot diagram, the invariant $W(D)$ is an honest Laurent polynomial. Using the $W$-polynomial Eisermann and Lamm show in~\cite{Ei.La11} that the diagram for $8_9$ in Figure~\ref{f:89} is not weakly symmetrically equivalent to the one 
obtained by switching crossings on the axis, and they exhibit an infinite family of pairs of symmetric union $2$-bridge knot diagrams $(D_n, D'_n)$ such that $D_n$ and $D'_n$ are Reidemeister equivalent 
but not weakly symmetrically equivalent for $n = 3$ and $n\geq 5$. The diagrams $D_4$ and $D'_4$, representing the knot $10_{42}$ are those shown in Figure~\ref{f:89}. They have the same $W$-polynomial, so the question of their (weak) symmetric equivalence was left unanswered. 

\subsection{Results and contents of the paper}\label{ss:results}

Our main result is Theorem~\ref{t:refined-invariance}, stating that (i) every topological spin model~\cite{Jo89} gives rise to infinitely many invariants of symmetric equivalence and (ii) 
such invariants satisfying a certain extra condition are in fact invariants of weak symmetric equivalence. As we point out in Remark~\ref{r:refined-typeII-exist}, each topological spin model gives rise in this way to at least four (essentially equivalent) invariants of weak symmetric equivalence. 

We give the following three applications of Theorem~\ref{t:refined-invariance}. $(1)$ Let $D_{10_{42}}$ (respectively $D'_{10_{42}}$) be the central (respectively the right-most) symmetric union diagram of Figure~\ref{f:89}. We prove that $D_{10_{42}}$ and $D'_{10_{42}}$ are not symmetrically equivalent, providing a partial answer to a question left open by Eisermann and Lamm~\cite[\S 6.4]{Ei.La11}. 
$(2)$ Let $D_{8_9}$ be the left-most diagram of Figure~\ref{f:89}, and let $D'_{8_9}$ be the diagram obtained from $D_{8_9}$ by switching all the crossings on the axis. As we explained in the paragraph immediately following Theorem~\ref{t:sRm}, the two diagrams $D_{8_9}$ and $D'_{8_9}$ are Reidemeister equivalent. We use Theorem~2.4 to prove that $D_{8_9}$ and $D'_{8_9}$ 
are not weakly symmetrically equivalent. $(3)$ We apply a gluing formula in conjunction with Theorem~\ref{t:refined-invariance} to construct, 
for each $n\geq 1$, symmetrically non-equivalent symmetric union diagrams of the connected sum of $n$ copies of $10_{42}$, as well as weakly symmetrically non-equivalent symmetric union diagrams of the connected sum of $n$ copies of $8_9$.

Section~\ref{s:spin-models} contains the necessary background material and the statement of Theorem~\ref{t:refined-invariance}. Section~\ref{s:invariance} contains the proof of Theorem~\ref{t:refined-invariance}. Section~\ref{s:applications} contains three applications of Theorem~\ref{t:refined-invariance}, and Section~\ref{s:gluing} the proof of the gluing formula.

\subsection*{Acknowledgements} The first author was partially supported by an Indam grant, and hosted by the IMT in Toulouse, during the early stages of this paper. The first author wishes to thank Francesco Costantino and the IMT for the hospitality. The authors are grateful to the anonymous referee for helpful comments and suggestions.

\section{Spin models and their refinements}\label{s:spin-models}

\subsection{Spin models}\label{ss:spinmodels}
We recall the theory of topological spin models for links in $S^3$ as introduced in~\cite{Jo89}. Let $X = \{1,2,\ldots, n\}$, $n\geq 2$, denote by $\Mat$ the space of square $n\x n$ complex matrices whose rows and columns are indexed by elements of the set $X$, and let $d\in\{\pm\sqrt{n}\}$. Given a symmetric, 
complex matrix $W^+\in \Mat$ with nonzero entries, let $W^-\in \Mat$ be the matrix uniquely determined by 
the equation 
\begin{equation}\label{e:type-II}
W^+\circ W^- = J,
\end{equation}
where $\circ$ is the Hadamard, i.e.~entry-wise, product and $J$ is the all-$1$ matrix.
Define, for each matrix $A\in \Mat$ and $a,b\in X$, the vector $Y^A_{ab}\in\bC^X$ by setting 
\[
Y^A_{ab}(x) := \frac {A(x,a)}{A(x,b)}\in\bC,\quad x\in X.  
\]
Then, the pair $M=(W^+,d)$ is a {\em spin model} if the following equations hold:
\begin{equation}\label{e:type-III}
W^+ Y^{W^+}_{ab} = d W^-(a,b) Y^{W^+}_{ab}\quad\text{for every $a,b\in X$}.
\end{equation} 
Observe that, since $Y^{W^+}_{aa}$ is the all-$1$ vector for each $a\in X$, 
taking $b=a$ in Equation~\eqref{e:type-III} gives  
\begin{equation}\label{e:type-I}
\dfrac1d\sum_{x\in X} W^+(y,x) = W^-(a,a)\quad\text{for every $y,a\in X$.}
\end{equation}
In particular, $W^-(a,a)$ and therefore the {\em modulus} $\al_W = W^+(a,a)=1/W^-(a,a)\in\bC$ 
of the spin model, are independent of $a\in X$. 

\begin{exas}
$(1)$ Let $n\geq 2$ be an integer and $d\in\{\pm \sqrt{n}\}$. Let $\xi\in\bC\setminus\{0\}$ 
be one of the four complex numbers such that $d = -\xi^2-\xi^{-2}$. Then, setting  
\[
W^+_{\rm Potts} = (-\xi^{-3})I + \xi (J-I),
\]
the pair $(W^+_{\rm Potts},d)$ is the well--known {\em Potts model} introduced in~\cite{Jo89}.
	
\noindent
$(2)$ Let $d = \sqrt{5}$, $\om = e^{2\pi i/5}$ and  
\[
	W^+_{\rm pent} = \mat{
		1&\om &\om^{-1}&\om^{-1} & \om\\
		\om&1&\om&\om^{-1}&\om^{-1}\\
		\om^{-1}&\om&1&\om&\om^{-1}\\
		\om^{-1}&\om^{-1}&\om& 1&\om\\
		\om&\om^{-1}&\om^{-1}&\om&1}.
\]
Then, $(W^+_{\rm pent},d)$ is one of the spin models studied in~\cite{Go.Jo89}
and mentioned in~\cite{Jo89, Jo89a}. We shall call it the {\em pentagonal model}, 
like the `rescaled' version $(-iW^+_{\rm pent},-\sqrt{5})$ considered in~\cite{de94}.
\end{exas} 

A spin model $M=(W^+,d)$ defines a link invariant as follows. Let $D\subset\bR^2$ be a connected diagram of an oriented link. Let $\Ga_D$ be the planar, signed medial 
graph associated to the black regions of any checkerboard coloring of $\bR^2\setminus D$. 
Let $\Ga_D^0$, $\Ga_D^1$ be the sets of vertices, respectively edges of $\Ga_D$ and let $N=|\Ga_D^0|$. Given $e\in\Ga^1$, we denote by $v_e$ and $w_e$ (in any order) the vertices of $e$.  Define the {\em partition function} $Z_M(D)\in\bC$ by 
\[
Z_M(D) = d^{-N}\sum_{\si\co\Ga^0_D\to X} \prod_{e\in\Ga_D^1} W^{s(e)}(\si(v_e),\si(w_e)),
\]
where the sum is taken over the set of all maps $\si$ from $\Ga^0_D$ to $X$, and $s(e)\in\{+,-\}$ is the sign of the edge $e$. Let the \emph{normalized partition function} $I_M(D)$ be
\[
I_M(D) := \al_W^{-w(D)} Z_M(D), 
\]
where $w(D)$ is the writhe of $D$. When $D$ is not connected, we define both 
$Z_M(D)$ and $I_M(D)$ as the product of the values 
of $Z_M$ and, respectively, $I_M$ on its connected components. 
\begin{thm}[\cite{Jo89}]\label{t:invariance}
	Let $M=(W^+,d)$ be a spin model and $D\subset\bR^2$ a connected, oriented 
	link diagram. Then, (i) $I_M(D)$ is independent of the choice of coloring and 
	(ii) $I_M (D) = I_M (D')$ for every link diagram $D'$ Reidemeister equivalent to $D$.
	\qed\end{thm}

\subsection{Refined spin models}\label{ss:refinedspinmodels}

Our idea is to refine the definition of a topological spin model by taking into account the presence 
of the axis, in the spirit of the refined Jones polynomial of Subsection~\ref{ss:rJp}.
Let $D\subset\bR^2$ be an oriented link diagram transverse to the axis 
$B = \{0\}\x\bR$. Since $B$ goes through some of the crossings of $D$, for any choice 
of a checkerboard coloring of $\bR^2\setminus D$, the corresponding medial 
graph $\Ga_D$ acquires some distinguished edges. We are going to assign suitably chosen 
weights to such distinguished edges. 

Let $X=\{1,\ldots, n\}$ with $n\geq 2$, and let $(W^+,d)$ be a spin model with $W^+\in \Mat$. 
Recall from Subsection~\ref{ss:spinmodels} that the matrix $W^+$ determines the vectors 
$Y^W_{ab}\in\bC^X$, $a,b\in X$. Nomura~\cite{No97} showed that the set $N_W\subset \Mat$ 
of matrices which have the vectors $Y^W_{ab}$ as eigenvectors is a commutative algebra with 
respect to both the ordinary matrix product and the Hadamard product.  
$N_W$ is sometimes called the~\emph{Nomura algebra}. Clearly, Equations~\eqref{e:type-III} 
imply $W^+\in N_W$. Let $\psi\co N_W\to \Mat$ be the map defined by requiring that, for each $A\in N_W$, 
the matrix $\psi(A)$ satisfies 
\[
A Y^{W^+}_{ab} = \psi(A)(a,b) Y^{W^+}_{ab}\quad\text{for every $a,b\in X$}.
\]
We are going to use the following facts: (i) $N_W$ is closed under transposition and (ii) $N_W$ is \emph{self-dual}, which 
means that $\psi$ induces a linear isomorphism $\psi\co N_W\to N_W$ and $\psi^2 = n\tau$, 
where $\tau\co N_W\to N_W$ is the transposition map. For these facts, as well as for more information about the 
Nomura algebra, we refer the reader to~\cite{Ja.Ma.No98}.
Observe that Equation~\eqref{e:type-III} 
is equivalent to the equality $W^- = \psi(W^+)/d$, hence $W^-\in N_W$. 
More generally, given any matrix $A^+\in N_W$ we can define $A^- := \psi(A^+)/d\in N_W$. 
Then, it follows from $\psi^2 = n\tau$ and $d^2=n$ that $A^+ = \psi(A^-)/d$.
In the same way as Equation~\eqref{e:type-I} we deduce 
\begin{equation}\label{e:R1}
\dfrac1d\sum_{x\in X} A^+(y,x) = A^-(a,a)\quad\text{and}\quad
\dfrac1d\sum_{x\in X} A^-(y,x) = A^+(a,a)\quad\text{for every $y,a\in X$}.
\end{equation}
In particular, the complex numbers $\al_{A^+} := A^+(a,a)$ and $\al_{A^-} := A^-(a,a)$ are independent of $a\in X$.

\begin{defns}\label{d:refined-spin-model}
	A {\em refined spin model} is a triple $(W^+,V^+,d)$ such that: 
	\begin{itemize} 
		\item
		$(W^+,d)$ is a spin model; 
		\item
		$V^+$ is a symmetric matrix belonging to the Nomura algebra $N_W$; 
		\item 
		$\al_{V^+}\cdot \al_{V^-}\neq 0$. 
	\end{itemize}
A {\em refined spin model of type II} is a refined spin model $(W^+,V^+,d)$ such that 
$V^+$ is a {\em type II} matrix, i.e.~such that $V^+\circ V^- = J$.
\end{defns} 

\begin{rem}\label{r:refined-typeII-exist} 
Every spin model $(W^+,d)$ admits a refinement $(W^+,V^+,d)$ of type II. Indeed, by definition $I\in N_W$, therefore $J=\psi(I)\in N_W$. Thus, if $\xi\in\bC\setminus\{0\}$ is one of the four complex numbers such that $d=-\xi^2-\xi^{-2}$, the symmetric, type II matrix $(-\xi^{-3}) I + \xi (J-I)\in N_W$ can be chosen as $V^+$. In other words, each spin model $(W^+,d)$ admits four type II refinements of the form $(W^+, V^+_{\rm Potts}, d)$, where $V^+_{\rm Potts} = (-\xi^{-3}) I + \xi (J-I)$. Note that $(V^+_{\rm Potts}, d)$ is a Potts model. Refined spin models of the form $(W^+, V^+_{\rm Potts}, d)$ will be referred to as {\em Potts-refined spin models}. 
\end{rem}

Let $\widehat M=(W^+,V^+,d)$ be a a refined spin model, $D$ an oriented, symmetric link diagram, 
and $c$ a checkerboard coloring of $\bR^2\setminus D$. 
Let $\Ga_D$ be the planar, signed medial graph associated to the black regions of $c$. 
The set $\Ga_D^1$ of the edges of $\Ga_D$ contains the set $\Ga_B^1$ of edges corresponding to crossings on the axis.  
We define the {\em partition function} $Z_{\widehat M}(D,c)$ by the formula
\[
Z_{\widehat M}(D,c) := d^{-N} \sum_{\si\co\Ga^0_D\to X} 
\prod_{e\in\Ga_B^1} V^{s(e)}(\si(v_e),\si(w_e))\prod_{e\in\Ga_D^1\setminus \Ga_B^1} W^{s(e)}(\si(v_e),\si(w_e)), 
\]
where $s(e)\in\{+,-\}$ is the sign of the edge $e$, and the 
{\em normalized partition function} $I_{\widehat M}(D,c)$ by 
\[
I_{\widehat M}(D,c) :=  \al_{V^+}^{-p_B(D)}\al_{V^-}^{-n_B(D)} Z_{\widehat M}(D,c),
\]
where $p_B(D)$ and $n_B(D)$ denote, respectively, the numbers of positive and negative 
crossings on the axis. 
As  in the case of the ordinary spin models, when $D$ is not connected we define both 
$Z_{\widehat M}(D,c)$ and $I_{\widehat M}(D,c)$ as the product of the values of $Z_{\widehat M}$ 
and, respectively, $I_{\widehat M}$ on its connected components with the induced colorings. 

We are ready to state our main result. Its proof will be given in the next section. 

\begin{thm}\label{t:refined-invariance}
Let $\widehat M$ be a refined spin model and $D_i\subset\bR^2$, $i=1,2$ two oriented, 
symmetrically equivalent symmetric (with respect to the axis $B$) union diagrams. Then, for any choice of checkerboard colorings 
$c_i$ of $\bR^2\setminus D_i$, we have 
\begin{equation}\label{e:se-invariance}
I_{\widehat M}(D_1,c_1) = I_{\widehat M}(D_2,c_2).
\end{equation}
Moreover, if $\widehat M$ is of type II then~\eqref{e:se-invariance} 
holds if $D_1$ and $D_2$ are weakly symmetrically equivalent. 
\end{thm}

\section{Proof of Theorem~\ref{t:refined-invariance}}\label{s:invariance}

Throughout the section we denote by $\widehat{M}$  
a fixed refined spin model $(W^+,V^+,d)$ and by $M$ its underlying spin model $(W^+,d)$.

\subsection{Invariance under the $sR$ and the $S2(h)$ moves}\label{ss:sRmS2hinv}

\begin{prop}\label{p:R1S2(h)inv}
	Let $\widehat{M}$ be a refined spin model and $(D,c)$ and $(D',c')$ two colored and 
	oriented symmetric link diagrams. If $(D',c')$ is obtained from $(D,c)$ by applying 
	either an $S2(h)$ move or a symmetric Reidemeister move off the axis, then
	\[
	I_{\widehat{M}}(D,c) = I_{\widehat{M}}(D',c').
	\]
\end{prop}

\begin{proof}
	An $S2(h)$-move does not change the edges of the medial graph $\Ga_D$ corresponding 
	to crossings on the axis, therefore the equality $I_{\widehat{M}}(D',c') = I_{\widehat{M}}(D,c)$ 
	holds for the same reason as the equality $I_M(D) = I_M(D')$ (cf.~\cite{Jo89, de94}). 
	A similar argument applies for a symmetric Reidemeister move off the axis. 
\end{proof}

\subsection{Invariance under the $S3$ and the $S2(\pm)$ moves}\label{ss:S3-S2inv}

Suppose that the medial graphs $\Ga_D$ and $\Ga_{D'}$ of two colored, oriented, 
symmetric link diagrams $(D,c)$ and $(D',c')$ appear locally as in 
Figure~\ref{f:ST1} and coincide elsewhere. The dashed arrows represent edges corresponding 
to crossings on the axis -- let us ignore the vertex labels for the moment.  
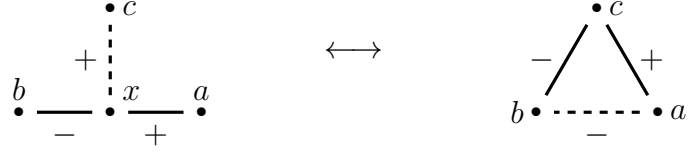
\begin{figure}[ht]
	\centering
	\begin{tikzpicture}[scale=0.8, every path/.style ={very thick}] 
\node (a) at (1.5,0) {\tiny$\bullet$};
\node (b) at (-1.5,0) {\tiny$\bullet$};
\node (c) at (0,1.732) {\tiny$\bullet$};
\node (x) at (0,0) {\tiny$\bullet$};

\path[dashed] (x) edge node[left] {$+$} (c);
\path (x) edge node[below] {$+$}(a);
\path (b) edge node[below] {$-$} (x);

\node[above right] at (x) {$x$};
\node[above] at (a) {$a$};
\node[right] at (c) {$c$};
\node[above] at (b) {$b$};

\begin{scope}[xshift = 4cm]
\node (arrow) at (0,1) {$\longleftrightarrow$};
\end{scope}

\begin{scope}[xshift = 8cm]
\node (a) at (1,0) {\tiny$\bullet$};
\node (b) at (-1,0) {\tiny$\bullet$};
\node (c) at (0,1.732) {\tiny$\bullet$};

\path[dashed] (a) edge node[below] {$-$} (b);
\path (c) edge node[right] {$+$} (a);
\path (b) edge node[left] {$-$} (c);

\node[right] at (a) {$a$};
\node[right] at (c) {$c$};
\node[left] at (b) {$b$};

\end{scope}
\end{tikzpicture}
	\caption{The directed medial graphs $\Ga_D$ (left) and $\Ga_{D'}$ (right)}
	\label{f:ST1}
\end{figure}
Then, we claim that the equality $I_{\widehat{M}}(D,c) = I_{\widehat{M}}(D',c')$ holds 
for each refined spin model $\widehat{M}$. As explained in Subsection~\ref{ss:spinmodels}, 
we have $\psi(V^+) = d V^-$ if and only if $V^+ Y^{W^+}_{ab} = d V^-(a,b) Y^{W^+}_{ab}$
for every $a,b\in X$. More explicitly,  
\begin{equation}\label{e:Star-triangle-1}
\sum_{x\in X} V^+(x,c)W^{+}(x,a)W^{-}(b,x) = d V^{-}(a,b)W^{+}(c,a)W^-(b,c),
\quad\text{for each}\quad a,b,c\in X.
\end{equation}
As the labels in Figure~\ref{f:ST1} show, Equations~\eqref{e:Star-triangle-1} 
guarantee that the different local contributions to the 
normalized partition functions for $D$ and $D'$ coincide. 
Note that, although the three vertices labeled $a,b$ and $c$ are drawn as if they were distinct, 
the equality $I_{\widehat{M}}(D,c) = I_{\widehat{M}}(D',c')$ still holds if two 
of them coincide. 

All possible instances of locally different medial graphs with the same normalized partition functions are displayed in Figure~\ref{f:ST2}, where $\ep_1,\ep_2\in\{\pm\}$. 
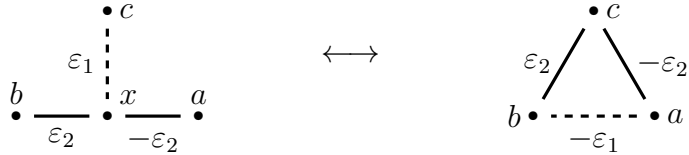
\begin{figure}[ht]
	\centering 
	\begin{tikzpicture}[scale=0.8, every path/.style ={very thick}] 
\node (a) at (1.5,0) {\tiny$\bullet$};
\node (b) at (-1.5,0) {\tiny$\bullet$};
\node (c) at (0,1.732) {\tiny$\bullet$};
\node (x) at (0,0) {\tiny$\bullet$};

\path[dashed] (c) edge node[left] {$\ep_1$} (x);
\path (x) edge node[below] {$-\ep_2$}(a);
\path (b) edge node[below] {$\ep_2$} (x);

\node[above right] at (x) {$x$};
\node[above] at (a) {$a$};
\node[right] at (c) {$c$};
\node[above] at (b) {$b$};

\begin{scope}[xshift = 4cm]
\node (arrow) at (0,1) {$\longleftrightarrow$};
\end{scope}

\begin{scope}[xshift = 8cm]
\node (a) at (1,0) {\tiny$\bullet$};
\node (b) at (-1,0) {\tiny$\bullet$};
\node (c) at (0,1.732) {\tiny$\bullet$};

\path[dashed] (a) edge node[below] {$-\ep_1$} (b);
\path (c) edge node[right] {$-\ep_2$} (a);
\path (b) edge node[left] {$\ep_2$} (c);

\node[right] at (a) {$a$};
\node[right] at (c) {$c$};
\node[left] at (b) {$b$};
\end{scope}

\end{tikzpicture}
	\caption{All the star-triangle identities in graphical form.}
	\label{f:ST2}
\end{figure}
We will make use of them in Subsections~\ref{ss:S3-S2inv} and~\ref{ss:S4inv}. 
The previous remark about the vertices labeled $a,b$ and $c$ applies. 
Following standard terminology, we shall call {\em star-triangle identities} the identities in Figure~\ref{f:ST2}. The reason why such identities hold is the following. As explained in Subsection~\ref{ss:refinedspinmodels}, the equality $\psi(V^+)=dV^-$ implies that 
$\psi(V^-) = dV^+$, which is equivalent to saying that $V^- Y^{W^+}_{ab} = d V^+(a,b) Y^{W^+}_{ab}$
for every $a,b\in X$. Moreover, since $Y^{W^+}_{ab} = Y^{W^-}_{ba}$ for every $a,b\in X$, 
we also have $V^+ Y^{W^-}_{ab} = d V^-(a,b) Y^{W^-}_{ab}$ and 
$V^- Y^{W^-}_{ab} = d V^+(a,b) Y^{W^-}_{ab}$ for every $a,b\in X$. One can now easily 
check that these equations imply the identities of Figure~\ref{f:ST2}. 

The following remark will be used in Subsection~\ref{ss:S4inv}. 
\begin{rmk}\label{r:star-triangle-graph}
	The algebraic identities represented by the graphs of Figure~\ref{f:ST2} hold for 
	the normalized partition function (defined in the obvious way) of any signed 
	graph $\Ga$ with some distinguished edges. In particular, $\Ga$ does not 
	need to be the medial graph of a diagram transverse to the axis.  
\end{rmk}

\begin{prop}\label{p:S3inv}
	Let $\widehat{M}$ be a refined spin model and $(D,c)$, $(D',c')$ two colored, 
	oriented symmetric union link diagrams. If $(D',c')$ 
	is obtained from $(D,c)$ by applying a symmetric Reidemeister 
	move of type $S3(o\pm)$ or $S3(u\pm)$, then
	\[
	I_{\widehat{M}}(D,c) = I_{\widehat{M}}(D',c').
	\]
\end{prop}

\begin{proof}
	The possible local changes of a colored symmetric union diagram are obtained from the one 
	shown in Figure~\ref{f:S3} by mirroring the picture or rotating it by $180^0$ around the 
	$x$, $y$ or $z$ axes. 
	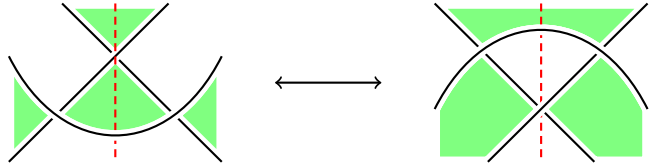
\begin{figure}[ht]
		\begin{tikzpicture}[scale=1.4] 

\draw[fill, green, opacity =.5] (-1,-.5) --  (-1,.5) .. controls +(.5,-1) and +(-.5,-1) .. (1,.5) -- (1,-.5) --  (-.5,1) --  (.5,1) -- (-1,-.5);

\pgfsetlinewidth{10*\pgflinewidth}
\draw[white] (-1,-.5) --  (-1,.5) .. controls +(.5,-1) and +(-.5,-1) .. (1,.5) -- (1,-.5) --  (-.5,1) --  (.5,1) -- (-1,-.5);
\pgfsetlinewidth{.2*\pgflinewidth}

\draw  (1,-.5) --  (-.5,1);

\pgfsetlinewidth{5*\pgflinewidth}
\draw[white](-1,-.5) --  (.5,1);
\pgfsetlinewidth{.2*\pgflinewidth}
\draw  (-1,-.5) --  (.5,1);

\draw[densely dashed, red] (0,1) -- (0,-0.5);

\pgfsetlinewidth{5*\pgflinewidth}
\draw[white](-1,.5) .. controls +(.5,-1) and +(-.5,-1) .. (1,.5);
\pgfsetlinewidth{.2*\pgflinewidth}
\draw (-1,.5) .. controls +(.5,-1) and +(-.5,-1) .. (1,.5);

\draw[<->] (1.5,0.25) -- (2.5,0.25);

\draw[fill, green, opacity =.5] (3,1) --  (5,1) -- (3.5,-.5) -- (3,-.5) -- (3,0) .. controls +(.5,1) and +(-.5,1) .. (5,0) --  (5,-.5) -- (4.5,-.5) -- (3.5,.5) -- cycle;

\draw[fill, white] (4,0) -- (4.56,.56) .. controls +(-.25,.25) and +(.25,.25) .. (3.44,.56) -- cycle;

\pgfsetlinewidth{5*\pgflinewidth}
\draw[white] (3,1) --  (5,1) -- (3.5,-.5) -- (3,-.5) -- (3,0) .. controls +(.5,1) and +(-.5,1) .. (5,0) --  (5,-.5) -- (4.5,-.5) -- (3.5,.5) -- cycle;
\draw[white]  (3,0) .. controls +(.5,1) and +(-.5,1) .. (5,0);
\pgfsetlinewidth{.2*\pgflinewidth}

\draw  (4.5,-.5) -- (3,1);

\draw[densely dashed, red] (4,1) -- (4,-0.5);

\pgfsetlinewidth{5*\pgflinewidth}
\draw[white](5,1) -- (3.5,-.5);
\pgfsetlinewidth{.2*\pgflinewidth}
\draw  (5,1) -- (3.5,-.5);
\pgfsetlinewidth{5*\pgflinewidth}
\draw[white] (3,0) .. controls +(.5,1) and +(-.5,1) .. (5,0);
\pgfsetlinewidth{.2*\pgflinewidth}
\draw (3,0) .. controls +(.5,1) and +(-.5,1) .. (5,0);

\end{tikzpicture}
		\caption{Local changes induced by $S3$ moves}
		\label{f:S3}
	\end{figure}
	It is a straightforward exercise to check that all the changes of the corresponding medial graphs
	are included among the ones described by Figure~\ref{f:ST2}. This immediately implies the statement. 
\end{proof} 

\begin{cor}\label{c:color-indep}
	Let $\widehat{M}$ be a refined spin model and $D$ an oriented, symmetric union link diagram. 
	Given distinct colorings $c$ and $c'$, we have 
	\[
	I_{\widehat{M}}(D,c) = I_{\widehat{M}}(D,c').
	\]
\end{cor}

\begin{proof}
	The proof we give is similar in spirit to the proof of~\cite[Proposition~2.14]{Jo89}.  
	Since $D$ is a symmetric union diagram, at least one strand of $D$ intersects the axis away from 
	the crossings. Applying a sequence of symmetric Reidemeister moves, $S2(h)$ and $S3$ moves 
	we can shift that strand ``downwards'' without changing $I_{\widehat{M}}$. Hence, 
	we assume without loss of generality that $(D,c)$ looks like the colored diagram shown 
	in the second picture from the left in Figure~\ref{f:duality}. Note that the medial graph 
	of $(D,c)$ coincides with the medial graph of the leftmost diagram in Figure~\ref{f:duality}. 
	Another sequence of symmetric Reidemeister moves, $S2(h)$ and $S3$ moves as suggested in the 
	remaining pictures of Figure~\ref{f:duality} turns $(D,c)$, without altering $I_{\widehat{M}}$, 
	into the right-most diagram of Figure~\ref{f:duality}, which has the same medial graph 
	as $(D,c')$.  
\end{proof}

\begin{figure}[ht]
	\begin{tikzpicture}[scale = 0.75, every path/.style ={very thick}]
\draw (-1,-1) rectangle (1,1);
\draw[dashed, red] (0,-2) -- (0,1.5);
\node at (0,0) {$D$};
\draw[pattern color = red, pattern = north east lines] (-.75,-1) .. controls +(0,-.5) and +(0,-.5)  .. (.75,-1) -- cycle;
\draw (-.75,-1) .. controls +(0,-.5) and +(0,-.5) .. (.75,-1) -- cycle;


\begin{scope}[xshift=0.5cm]
\node at (1.8,0){$\longleftrightarrow$};
\draw[pattern color = red, pattern = north east lines] (3.25,-1) -- (3.25,-1.5) .. controls +(-1.5,0) and +(0,-1.5) .. (3.25,-1.5) -- (4.75,-1.5) .. controls +(1.5,0) and +(0,-1.5) .. (4.75,-1.5)  -- (4.75,-1)-- cycle;

\draw[white, very thick] (3.25,-1) -- (3.25,-1.5) .. controls +(-1.5,0) and +(0,-1.5) .. (3.25,-1.5) -- (4.75,-1.5) .. controls +(1.5,0) and +(0,-1.5) .. (4.75,-1.5)  -- (4.75,-1)-- cycle;

\draw (3.25,-1) -- (3.25,-1.5) .. controls +(-1.5,0) and +(0,-1.5) .. (3.25,-1.5) -- (4.75,-1.5) .. controls +(1.5,0) and +(0,-1.5) .. (4.75,-1.5)  -- (4.75,-1)-- cycle;
\draw (3,-1) rectangle (5,1);
\draw[dashed, red] (4,-2) -- (4,1.5);
\node at (4,0) {$D$};
\end{scope}


\node at (6.5,0){$\longleftrightarrow$};
\draw[pattern = north east lines,pattern color = red ] (8.25,-1) -- (8.25,-1.5) .. controls +(0,-1.5)  and +(-1.5,0).. (8.25,1.25) -- (9.75,1.25) .. controls +(1.5,0) and +(0,-1.5) .. (9.75,-1.5)  -- (9.75,-1)-- cycle;

\draw[white, very thick] (8.25,-1) -- (8.25,-1.5) .. controls +(0,-1.5)  and +(-1.5,0).. (8.25,1.25) -- (9.75,1.25) .. controls +(1.5,0) and +(0,-1.5) .. (9.75,-1.5)  -- (9.75,-1)-- cycle;

\draw (8.25,-1) -- (8.25,-1.5) .. controls +(0,-1.5)  and +(-1.5,0).. (8.25,1.25) -- (9.75,1.25) .. controls +(1.5,0) and +(0,-1.5) .. (9.75,-1.5)  -- (9.75,-1)-- cycle;

\draw[fill, white] (8,-1) rectangle (10,1);
\draw (8,-1) rectangle (10,1);
\draw[dashed, red] (9,-2) -- (9,1.5);
\node at (9,0) {$D$};


\begin{scope}[shift = {+(-15,0)}]
\draw[pattern color = red, pattern = north east lines] (7,-2) rectangle (11,1.5);
\draw[white] (7,-2) rectangle (11,1.5);
\draw[fill, white] (8.25,-1) -- (8.25,-1.5) .. controls +(0,-1.5)  and +(-1.5,0).. (8.25,1.25) -- (9.75,1.25) .. controls +(1.5,0) and +(0,-1.5) .. (9.75,-1.5)  -- (9.75,-1)-- cycle;

\draw[white, very thick] (8.25,-1) -- (8.25,-1.5) .. controls +(0,-1.5)  and +(-1.5,0).. (8.25,1.25) -- (9.75,1.25) .. controls +(1.5,0) and +(0,-1.5) .. (9.75,-1.5)  -- (9.75,-1)-- cycle;

\draw (8.25,-1) -- (8.25,-1.5) .. controls +(0,-1.5)  and +(-1.5,0).. (8.25,1.25) -- (9.75,1.25) .. controls +(1.5,0) and +(0,-1.5) .. (9.75,-1.5)  -- (9.75,-1)-- cycle;

\draw[fill, white] (8,-1) rectangle (10,1);
\draw (8,-1) rectangle (10,1);
\draw[dashed, red] (9,-2) -- (9,1.5);
\node at (9,0) {$D$};
\end{scope}
\end{tikzpicture} 
	\caption{Independence of $I_{\widehat{M}}$ from the choice of coloring}
	\label{f:duality}
\end{figure}
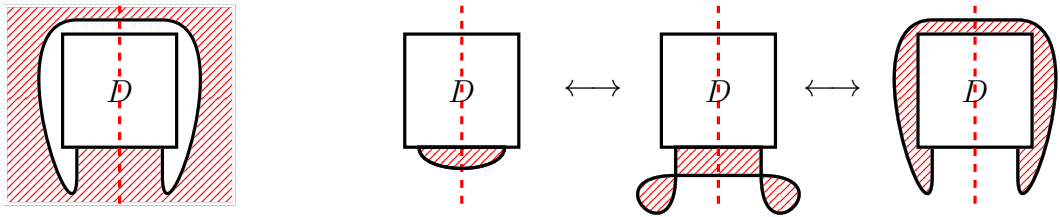

In view of Corollary~\ref{c:color-indep}, from now on we shall omit the coloring from the 
notation for the normalized partition function of symmetric union diagrams.  

\begin{prop}\label{p:S2(pm)}
	Let $D$ and $D'$ be two oriented, symmetric union link diagrams. If $D'$ is obtained 
	from $D$ by applying a symmetric Reidemeister move of type $S2(+)$ or $S2(-)$, then 
	\[
	I_{\widehat{M}}(D') = I_{\widehat{M}}(D).
	\]
\end{prop}

\begin{proof}
	By Corollary~\ref{c:color-indep}, it suffices to prove the statement for any choice of coloring. 
	The case of an $S2(-)$ move is illustrated in Figure~\ref{f:S2(-)}. 
	\begin{figure}[ht]
		\begin{tikzpicture}[scale=0.8, every path/.style ={very thick}, inner sep=2pt]

\draw[pattern = north east lines] (-1.5,-.5) .. controls +(1,0) and +(-1,0) .. (1.5,.5) -- (1.5,1.5) -- (-1.5,1.5) -- (-1.5,.5) .. controls +(1,0) and +(-1,0) .. (1.5,-.5) -- (1.5,-1.5) -- (-1.5,-1.5) -- cycle ;
\pgfsetlinewidth{10*\pgflinewidth}
\draw[white] (-1.5,-.5) .. controls +(1,0) and +(-1,0) .. (1.5,.5) -- (1.5,1.5) -- (-1.5,1.5) -- (-1.5,.5) .. controls +(1,0) and +(-1,0) .. (1.5,-.5) -- (1.5,-1.5) -- (-1.5,-1.5) -- cycle ;
\pgfsetlinewidth{.1*\pgflinewidth}

\draw (-1.5,-.5) .. controls +(1,0) and +(-1,0) .. (1.5,.5);
\pgfsetlinewidth{10*\pgflinewidth}
\draw[white] (-1.5,.5) .. controls +(1,0) and +(-1,0) .. (1.5,-.5);
\pgfsetlinewidth{.1*\pgflinewidth}
\draw (-1.5,.5) .. controls +(1,0) and +(-1,0) .. (1.5,-.5);

\draw[white, fill, opacity= .5] (-2,-2) rectangle (2,2);

\node (a) at (0,1) {\tiny$\bullet$};
\node (b) at (0,-1) {\tiny$\bullet$};
\draw[densely dashed, very thick] (a) -- (b);
\node at (0.4,0) {$-$};

\draw[<->] (3,0) -- (4,0);

\begin{scope}[shift = {+(7,0)}]
\draw[pattern = north east lines] (-1.5,-.5)  .. controls +(1,0) and +(-1,1) .. (0,0)  .. controls +(1,-1)  and +(-1,0) .. (1.5,.5) -- (1.5,1.5) -- (-1.5,1.5) -- (-1.5,.5).. controls +(1,0) and +(-1,-1) .. (0,0)  .. controls +(1,1)  and +(-1,0) .. (1.5,-.5) -- (1.5,-1.5) -- (-1.5,-1.5) -- cycle ;
\pgfsetlinewidth{10*\pgflinewidth}
\draw[white] (-1.5,-.5)  .. controls +(1,0) and +(-1,1) .. (0,0)  .. controls +(1,-1)  and +(-1,0) .. (1.5,.5) -- (1.5,1.5) -- (-1.5,1.5) -- (-1.5,.5).. controls +(1,0) and +(-1,-1) .. (0,0)  .. controls +(1,1)  and +(-1,0) .. (1.5,-.5) -- (1.5,-1.5) -- (-1.5,-1.5) -- cycle ;
\pgfsetlinewidth{.1*\pgflinewidth}

\draw(-1.5,-.5) .. controls +(1,0) and +(-1,1) .. (0,0);
\pgfsetlinewidth{10*\pgflinewidth}
\draw[white] (-1.5,.5) .. controls +(1,0) and +(-1,-1) .. (0,0);

\pgfsetlinewidth{.1*\pgflinewidth}
\draw (-1.5,.5) .. controls +(1,0) and +(-1,-1) .. (0,0)  .. controls +(1,1)  and +(-1,0) .. (1.5,-.5);
\pgfsetlinewidth{10*\pgflinewidth}
\draw[white] (0,0)  .. controls +(1,-1)  and +(-1,0) ..(1.5,.5);
\pgfsetlinewidth{.1*\pgflinewidth}
\draw  (0,0)  .. controls +(1,-1)  and +(-1,0) ..(1.5,.5);
\draw[white, fill] (0,0) circle (.1);

\draw (.0866,-.0866) -- (-.0866,.0866);

\draw[white, fill, opacity= .5] (-2,-2) rectangle (2,2);

\node (a) at (0,1) {\tiny$\bullet$};
\node (b) at (0,-1) {\tiny$\bullet$};
\draw[dashed, very thick] (a) -- (b);
\node at (0.3,0) {$-$};

\draw (a) .. controls +(-1,-1) and +(-1,1) .. (b);
\draw (a) .. controls +(1,-1) and +(1,1) .. (b);
\node[] at (1.1,0) {$+$};
\node[] at (-1.1,0) {$-$};
\end{scope}
\end{tikzpicture} 
		\caption{Colored diagrams and medial graphs differing by an $S2(-)$ move.}
		\label{f:S2(-)}
	\end{figure}
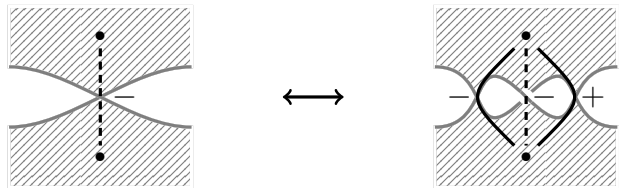
	The statement follows immediately from Equation~\eqref{e:type-II}. 
	The case of an $S2(+)$ is similar and left to the reader. 
\end{proof}

\subsection{Invariance under the $S1$ moves}\label{ss:S1inv}

\begin{prop}\label{p:S1inv}
	Let $\widehat{M}$ be  a refined spin model, and let $D$, $D'$ be two oriented, symmetric union link diagrams. If $D'$ is obtained from $D$ by applying an $S1$ move, then  
	\[
	I_{\widehat{M}}(D') = I_{\widehat{M}}(D).
	\]
\end{prop}

\begin{proof} 
	As in the proof of Proposition~\ref{p:S3inv}, the local change of a symmetric union 
	diagram due to a move of type $S1(-)$ is given, up to symmetries, by the left-hand portion of Figure~\ref{f:extrasRm}. In view of Corollary~\ref{c:color-indep}, the choice of coloring is irrelevant, so we make that choice so that the corresponding local change 
	of medial graphs is the one given by Figure~\ref{f:S1inv}. 
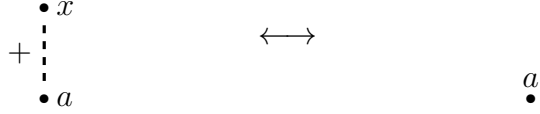
\begin{figure}[ht]
	\begin{tikzpicture}[scale=0.8, every path/.style ={very thick}] 
\node (x) at (0,1.5) {\tiny$\bullet$};
\node (a) at (0,0) {\tiny$\bullet$};

\path[dashed] (x) edge node[left] {$+$} (a);

\node[right] at (a) {$a$};
\node[right] at (x) {$x$};

\begin{scope}[xshift = 4cm]
\node (arrow) at (0,1) {$\longleftrightarrow$};
\end{scope}

\begin{scope}[xshift = 8cm]
\node (a) at (0,0) {\tiny$\bullet$};


\node[above] at (a) {$a$};

\end{scope}
\end{tikzpicture}
	\caption{Local change due to an $S1$ move.}
	\label{f:S1inv}
\end{figure}
Suppose that $\Ga_{D'}$ is locally given by the left-hand side of Figure~\ref{f:S1inv},  
denote by $v_0$ the vertex labelled $a$ and by $e_0$ the dashed edge connecting $v_0$ to the vertex labeled $x$. Let $N = |\Ga^0_D|$ be the number of vertices of $\Ga_D$, so that $|\Ga^0_{D'}|=N+1$. By the definition of the partition function we have 
\begin{align*}
& Z_{\widehat{M}}(D')  = d^{-N-1} \sum_{\si\co\Ga^0_{D'}\to X} 
\prod_{e\in\Ga^1_B} V^{s(e)}(\si(v_e),\si(w_e))\prod_{e\in\Ga^1_{D'}\setminus \Ga^1_B} W^{s(e)}(\si(v_e),\si(w_e)) \\ 
& = d^{-N} \sum_{\si\co\Ga^0_{D}\to X} [\dfrac1d \sum_{x\in X} V^+(\si(v_0), x)] \prod_{e\in\Ga^1_B\setminus\{e_0\}} V^{s(e)}(\si(v_e),\si(w_e))\prod_{e\in\Ga^1_D\setminus\Ga^1_B} W^{s(e)}(\si(v_e),\si(w_e)) \\
& = \al_{V^-} Z_{\widehat{M}}(D),
\end{align*}
where the last equality is due to the fact that $\dfrac1d\sum_{x\in X} V^+(a,x) = \al_{V^-}$
for each $a\in X$, which follows from~\eqref{e:R1}. The equality 
$I_{\widehat{M}}(D') = I_{\widehat{M}}(D)$ now follows immediately from $p_B(D')=p_B(D)$ 
and $n_B(D')=n_B(D)+1$. The argument for an $S1(+)$ move is similar and left to the reader. 
\end{proof}

\subsection{Invariance under the $S4$ moves}\label{ss:S4inv}

In view of the results of Subsections~\ref{ss:sRmS2hinv}, \ref{ss:S3-S2inv} and~\ref{ss:S1inv}, the following concludes the proof of the first part of Theorem~\ref{t:refined-invariance}. 

\begin{prop}\label{p:S4inv}
Let $\widehat{M}$ be  a refined spin model, and let $D$, $D'$ be two oriented, symmetric union link diagrams. If $D'$ is obtained from $D$ by applying an $S4$ move, then  
	\[
	I_{\widehat{M}}(D') = I_{\widehat{M}}(D).
	\]
\end{prop}

\begin{proof}
	By Corollary~\ref{c:color-indep} we can choose an arbitrary coloring. We choose the 
	configuration of Figure~\ref{f:S4}. 
	\begin{figure}[ht]
		\begin{tikzpicture}[scale=0.25]

\draw[pattern = north east lines, pattern color  = red] (-10,5) -- (-2,5) .. controls +(1,0) and +(-1,0) .. (2,3) -- (10,3) -- (10,5)-- (2,5) .. controls +(-1,0) and +(1,0) .. (-2,3) -- (-10,3) -- cycle;

\draw[thick, color  = white] (-10,5) -- (-2,5) .. controls +(1,0) and +(-1,0) .. (2,3) -- (10,3) -- (10,5)-- (2,5) .. controls +(-1,0) and +(1,0) .. (-2,3) -- (-10,3) -- cycle;

\draw[thick] (-10,5) -- (-2,5) .. controls +(1,0) and +(-1,0) .. (2,3) -- (10,3);
\draw[thick] (10,5)-- (2,5) .. controls +(-1,0) and +(1,0) .. (-2,3) -- (-10,3);

\draw[pattern = north east lines, pattern color  = red] (-10,1) -- (-2,1) .. controls +(1,0) and +(-1,0) .. (2,-1) -- (10,-1) -- (10,1)-- (2,1) .. controls +(-1,0) and +(1,0) .. (-2,-1) -- (-10,-1) -- cycle;

\draw[thick, color  = white] (-10,1) -- (-2,1) .. controls +(1,0) and +(-1,0) .. (2,-1) -- (10,-1) -- (10,1)-- (2,1) .. controls +(-1,0) and +(1,0) .. (-2,-1) -- (-10,-1) -- cycle;

\draw[thick] (-10,1) -- (-2,1) .. controls +(1,0) and +(-1,0) .. (2,-1) -- (10,-1);
\draw[thick] (10,1)-- (2,1) .. controls +(-1,0) and +(1,0) .. (-2,-1) -- (-10,-1);

\node at (-0.8,-1.5) {$\pm$};
\node at (-0.8,2.5) {$\pm$};

\draw[dashed, red] (0,6) -- (0,-2);

\begin{scope}[xshift = 15cm]
\node at (0,2) {$\longleftrightarrow$};
\node at (0,3.2) {$S_4$};
\end{scope}

\begin{scope}[shift = {+(30,0)}]
\draw[pattern = north east lines, pattern color  = red] (-10,1) -- (-2,5) .. controls +(1,0) and +(-1,0) .. (2,3)  .. controls +(1,0) and +(-1,0) .. (10,-1) -- (10,1) .. controls +(-1,0) and +(1,0) .. (2,5) .. controls +(-1,0) and +(1,0) .. (-2,3)  .. controls +(-1,0) and +(1,0) .. (-10,-1) -- cycle;

\draw[thick, color  = white] (-10,1) -- (-2,5) .. controls +(1,0) and +(-1,0) .. (2,3)  .. controls +(1,0) and +(-1,0) .. (10,-1) -- (10,1) .. controls +(-1,0) and +(1,0) .. (2,5) .. controls +(-1,0) and +(1,0) .. (-2,3)  .. controls +(-1,0) and +(1,0) .. (-10,-1) -- cycle;

\draw[thick] (-10,1) -- (-2,5) .. controls +(1,0) and +(-1,0) .. (2,3)  .. controls +(1,0) and +(-1,0) .. (10,-1); 

\draw[thick] (10,1) .. controls +(-1,0) and +(1,0) .. (2,5) .. controls +(-1,0) and +(1,0) .. (-2,3)  .. controls +(-1,0) and +(1,0) .. (-10,-1);

\pgfsetlinewidth{10*\pgflinewidth}
\draw[white] (-10,5)  .. controls +(1,0) and +(-1,0) .. (-2,1) .. controls +(1,0) and +(-1,0) .. (2,-1)  .. controls +(1,0) and +(-1,0) .. (10,3) -- (10,5)  .. controls +(-1,0) and +(1,0) .. (2,1) .. controls +(-1,0) and +(1,0) .. (-2,-1) -- (-10,3) -- cycle;
\pgfsetlinewidth{.1*\pgflinewidth}

\draw[pattern = north east lines, pattern color  = red] (-10,5)  .. controls +(1,0) and +(-1,0) .. (-2,1) .. controls +(1,0) and +(-1,0) .. (2,-1)  .. controls +(1,0) and +(-1,0) .. (10,3) -- (10,5)  .. controls +(-1,0) and +(1,0) .. (2,1) .. controls +(-1,0) and +(1,0) .. (-2,-1) -- (-10,3) -- cycle;

\draw[thick, color  = white] (-10,5)  .. controls +(1,0) and +(-1,0) .. (-2,1) .. controls +(1,0) and +(-1,0) .. (2,-1)  .. controls +(1,0) and +(-1,0) .. (10,3) -- (10,5)  .. controls +(-1,0) and +(1,0) .. (2,1) .. controls +(-1,0) and +(1,0) .. (-2,-1) -- (-10,3) -- cycle;

\draw[thick] (-10,5)  .. controls +(1,0) and +(-1,0) .. (-2,1) .. controls +(1,0) and +(-1,0) .. (2,-1)  .. controls +(1,0) and +(-1,0) .. (10,3);

\draw[thick] (10,5)  .. controls +(-1,0) and +(1,0) .. (2,1) .. controls +(-1,0) and +(1,0) .. (-2,-1) -- (-10,3);

\draw[fill,white] (6,2.92) -- (7.5,2) -- (6,1.07) -- (4.4,2)-- cycle;

\draw[fill,white] (-6,2.92) -- (-4.5,2) -- (-6,1.07) -- (-7.8,2)-- cycle;

\node at (-0.8,-1.5) {$\pm$};
\node at (-0.8,2.5) {$\pm$};
\draw[dashed, red] (0,6) -- (0,-2);

\end{scope}

\end{tikzpicture}
		\caption{The choice of coloring for the $S4$ move.}
		\label{f:S4}
	\end{figure}
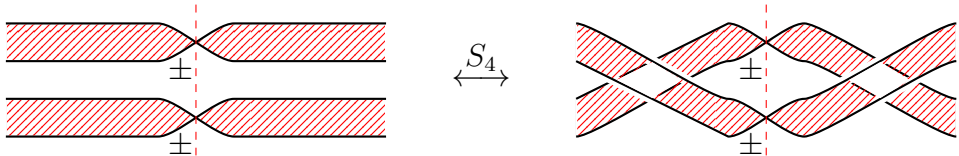
	There is a number of possible cases, depending on the types of crossings on the axis and 
	whether the two top strands go over or under the two bottom strands. 
	As illustrated in Figure~\ref{f:S4}, we now consider all configurations of crossings on the 
	axis simultaneously and we assume that the two top strands go over the two bottom strands.
	It is easy to check that the graphs $\Ga_1$ and $\Ga_2$ associated to the diagrams 
	of Figure~\ref{f:S4} are locally as in Figure~\ref{f:S4-graphs}, where 
	$\ep_1, \ep_2 \in \{ +,-\}$.
	\begin{figure}[ht]
		\begin{tikzpicture}[scale = 0.9, every path/.style ={very thick},
every node/.style={inner sep = 0pt}]
\node (a) at (-2,1) {\small$\bullet$};
\node (d) at (2,1)  {\small$\bullet$};
\node (b) at (-2,-1) {\small$\bullet$};
\node (c) at (2,-1)  {\small$\bullet$};
\node at (-2.4,1) {$a$};
\node at (2.4,1)  {$d$};
\node at (-2.4,-1) {$b$};
\node at (2.4,-1)  {$c$};
\draw[dashed] (d) -- (a);
\draw[dashed] (b) -- (c);

\node[above] at (0,1.2) {$\ep_1$};
\node[below] at (0,-1.2)  {$\ep_2$};

\node at (-3.5,0)  {$\Gamma_1 = $};

\begin{scope}[xshift = 9cm]
\node (a) at (-2,1) {\small$\bullet$};
\node (d) at (2,1)  {\small$\bullet$};
\node (b) at (-2,-1) {\small$\bullet$};
\node (c) at (2,-1)  {\small$\bullet$};
\node (x) at (-.75,1) {\small$\bullet$};
\node (t) at (.75,1)  {\small$\bullet$};
\node (y) at (-.75,-1) {\small$\bullet$};
\node (z) at (.75,-1)  {\small$\bullet$};
\node at (-2.4,1) {$a$};
\node at (2.4,1) {$d$};
\node at (-2.4,-1) {$b$};
\node at (2.4,-1)  {$c$};

\node at (-0.75,1.4) {$x$};
\node at (0.75,1.4)  {$t$};
\node at (-0.75,-1.4) {$y$};
\node at (0.75,-1.4)  {$z$};
\draw[dashed] (x) -- (t);
\draw[dashed] (z) -- (y);
\draw (x) -- (y);
\draw (z) -- (t);
\draw (x) -- (a);
\draw (z) -- (c);
\draw (b) -- (y);
\draw (d) -- (t);
\draw (b) -- (a);
\draw (d) -- (c);

\node at (0,1.3) {$\ep_2$};
\node at (0,-1.35)  {$\ep_1$};

\node[above] at (-1.375,1.1) {$+$};
\node[below] at (-1.375,-1.1)  {$+$};
\node[above] at (1.375,1.1) {$-$};
\node[below] at (1.375,-1.1)  {$-$};

\node at (1.7,0) {$+$};
\node at (0.45,0)  {$+$};
\node at (-2.3,0) {$-$};
\node at (-1.1,0)  {$-$};

\node at (-3.5,0)  {$\Gamma_2 = $};
\end{scope}
\end{tikzpicture}
		\caption{Local change of the medial graph under an $S4$ move}
		\label{f:S4-graphs}
	\end{figure}
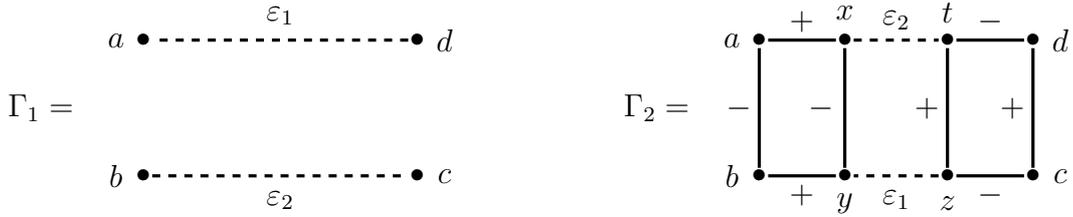
	As we now explain, Figure~\ref{f:S4-proof} contains the proof that the graphs $\Ga_1$ 
	and $\Ga_2$ have equal invariants $I_{\widehat M}$. 
	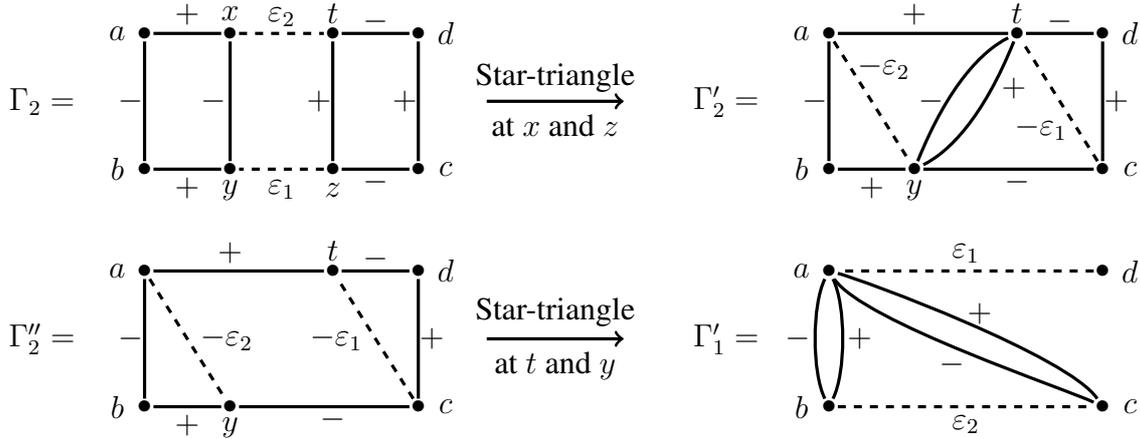
\begin{figure}[ht]
		\begin{tikzpicture}[scale = 0.9, every path/.style ={very thick},
every node/.style={inner sep = 0pt}]

\begin{scope}[yshift = 3.5cm]
\node (a) at (-2,1) {\small$\bullet$};
\node (d) at (2,1)  {\small$\bullet$};
\node (b) at (-2,-1) {\small$\bullet$};
\node (c) at (2,-1)  {\small$\bullet$};
\node (x) at (-.75,1) {\small$\bullet$};
\node (t) at (.75,1)  {\small$\bullet$};
\node (y) at (-.75,-1) {\small$\bullet$};
\node (z) at (.75,-1)  {\small$\bullet$};
\node at (-2.4,1) {$a$};
\node at (2.4,1)  {$d$};
\node at (-2.4,-1) {$b$};
\node at (2.4,-1)  {$c$};

\node at (-0.75,1.3) {$x$};
\node at (0.75,1.3)  {$t$};
\node at (-0.75,-1.3) {$y$};
\node at (0.75,-1.3)  {$z$};
\draw[dashed] (x) -- (t);
\draw[dashed] (z) -- (y);
\draw (x) -- (y);
\draw (z) -- (t);
\draw (x) -- (a);
\draw (z) -- (c);
\draw (b) -- (y);
\draw (d) -- (t);
\draw (b) -- (a);
\draw (d) -- (c);

\node at (0,1.25) {$\ep_2$};
\node at (0,-1.3)  {$\ep_1$};

\node[above] at (-1.375,1.1) {$+$};
\node[below] at (-1.375,-1.1)  {$+$};
\node[above] at (1.375,1) {$-$};
\node[below] at (1.375,-1)  {$-$};

\node[right] at (1.6,0) {$+$};
\node[left] at (.75,0)  {$+$};
\node[left] at (-2,0) {$-$};
\node[right] at (-1.2,0)  {$-$};

\draw[->] (3,0) -- (5,0);
\node at (4,.35) {Star-triangle};
\node at (4,-.35) {at $x$ and $z$};

\node at (-3.5,0)  {$\Gamma_2 = $};
\end{scope}

\begin{scope}[shift = {+(10,3.5)}]

\node (a) at (-2,1) {\small$\bullet$};
\node (d) at (2,1)  {\small$\bullet$};
\node (b) at (-2,-1) {\small$\bullet$};
\node (c) at (2,-1)  {\small$\bullet$};
\node (t) at (.75,1)  {\small$\bullet$};
\node (y) at (-.75,-1) {\small$\bullet$};
\node at (-2.4,1) {$a$};
\node at (2.4,1)  {$d$};
\node at (-2.4,-1) {$b$};
\node at (2.4,-1)  {$c$};

\node at (0.75,1.3)  {$t$};
\node at (-0.75,-1.3) {$y$};
\draw[dashed] (a) -- (y);

\draw (t) -- (a);
\draw (b) -- (y);
\draw (d) -- (t);
\draw (b) -- (a);
\draw (d) -- (c);

\draw (y) .. controls +(.5,.25) and +(-.5,-1.25) .. (t);
\draw (y) .. controls +(.5,1.25) and +(-.5,-.25) .. (t);
\draw (y) -- (c);
\draw[dashed] (c) -- (t);

\node[above] at (-.75,1.1) {$+$};
\node[below] at (-1.375,-1.1)  {$+$};
\node[above] at (1.375,1) {$-$};
\node[below] at (0.75,-1)  {$-$};

\node[right] at (2,0) {$+$};
\node[left] at (-.3,0)  {$-$};
\node[left] at (-2,0) {$-$};
\node[right] at (.5,0.2) {$+$};
\node at (-1.2,0.5) {$-\ep_2$};
\node at (1.1,-0.4) {$-\ep_1$};

\node at (-3.5,0)  {$\Gamma'_2 = $};
\end{scope}

\node (a) at (-2,1) {\small$\bullet$};
\node (d) at (2,1)  {\small$\bullet$};
\node (b) at (-2,-1) {\small$\bullet$};
\node (c) at (2,-1)  {\small$\bullet$};
\node (t) at (.75,1)  {\small$\bullet$};
\node (y) at (-.75,-1) {\small$\bullet$};
\node at (-2.4,1) {$a$};
\node at (2.4,1)  {$d$};
\node at (-2.4,-1) {$b$};
\node at (2.4,-1)  {$c$};

\node at (0.75,1.3)  {$t$};
\node at (-0.75,-1.3) {$y$};
\draw[dashed] (a) -- (y);

\draw (t) -- (a);
\draw (b) -- (y);
\draw (d) -- (t);
\draw (b) -- (a);
\draw (d) -- (c);

\draw (y) -- (c);
\draw[dashed] (c) -- (t);

\node[above] at (-.75,1.1) {$+$};
\node[below] at (-1.375,-1.1)  {$+$};
\node[above] at (1.375,1) {$-$};
\node[below] at (0.75,-1)  {$-$};

\node[right] at (2,0) {$+$};
\node[left] at (-2,0) {$-$};
\node at (-0.8,0) {$-\ep_2$};
\node at (0.8,0) {$-\ep_1$};

\draw[->] (3,0) -- (5,0);
\node at (4,.4) {Star-triangle};
\node at (4,-.4) {at $t$ and $y$};

\node at (-3.5,0)  {$\Gamma''_2 = $};

\begin{scope}[shift = {+(10,0)}]

\node (a) at (-2,1) {\small$\bullet$};
\node (d) at (2,1)  {\small$\bullet$};
\node (b) at (-2,-1) {\small$\bullet$};
\node (c) at (2,-1)  {\small$\bullet$};
\node at (-2.4,1) {$a$};
\node at (2.4,1)  {$d$};
\node at (-2.4,-1) {$b$};
\node at (2.4,-1)  {$c$};

\draw[dashed] (b) -- (c);

\draw (c) .. controls +(-.5,.75) and +(.5,-.1333) .. (a);
\draw (c) .. controls +(-.5,.25) and +(.5,-.75) .. (a);

\draw[dashed] (d) -- (a);

\node[above] at (0,1.1) {$\ep_1$};
\node[below] at (0,-1.1) {$\ep_2$};

\node[left] at (-2.25,0) {$-$};
\node[right] at (-1.75,0) {$+$};

\node[above right] at (0,0.2)  {$+$};
\node[below left] at (0,-0.2)  {$-$};
\draw (a) .. controls +(-.25,-.5) and +(-.25,.5) .. (b);
\draw (a) .. controls +(.25,-.5) and +(.25,.5) .. (b);

\node at (-3.5,0)  {$\Gamma'_1 = $};
\end{scope}
\end{tikzpicture}
		\caption{Invariance of $I_{\widehat{M}}$ under the $S4$ move}
		\label{f:S4-proof}
	\end{figure}
	Indeed, the upper part of Figure~\ref{f:S4-proof} describes the application to $\Ga_2$ of  
	two star-triangle identities from Figure~\ref{f:ST2},   
	resulting in the graph $\Ga'_2$. Note that, in view of Remark~\ref{r:star-triangle-graph},
	we do not need to keep track of the axis but only of the graphs and their distinguished 
	edges. Equation~\eqref{e:type-II} allows us to cancel 
	the two edges of $\Ga'_2$ connecting the vertices $t$ and $y$, 
	obtaining the graph $\Ga''_2$. The lower part of Figure~\ref{f:S4-proof} 
	shows how two more star-triangle identities can be applied to $\Ga''_2$  
	to obtain the graph $\Ga'_1$. After two more edge cancellations we get graph $\Ga_1$. 
	Observe that the vertices labeled 
	$a$ and $b$, as well as those labeled $c$ and $d$, are drawn as if they 
	were distinct, but the proof goes thorough if they coincide. 
	This concludes the argument in the cases when the two top strands go over the 
	two bottom strands. For the other cases 
	the argument is essentially the same, and therefore omitted. 
\end{proof} 

\subsection{Invariance under the $S2(v)$ moves}\label{ss:S2v}

The following result concludes the proof of the second part of Theorem~\ref{t:refined-invariance}. 

\begin{prop}\label{p:S2vinv}
	Let $\widehat{M}$ be  a refined spin model of type II, and let $D$, $D'$ be two oriented, symmetric union link diagrams. If $D'$ is obtained from $D$ by applying an $S2(v)$ move, then  
	\[
	I_{\widehat{M}}(D') = I_{\widehat{M}}(D).
	\]
\end{prop}

\begin{proof} 
The proof is very simple. Suppose that $D$ and $D'$ are the diagrams shown on the 
right-hand side of Figure~\ref{f:extrasRm}, with $D'$ having two more crossings 
on the axis. It is clear that we can choose 
the colorings so that $\Ga^0_{D'} = \Ga^0_D$ and $\Ga_{D'}$ 
has two more edges on the axis with opposite signs, connecting the same two vertices. 
The fact that $\widehat{M}$ is of type II implies that 
$Z_{\widehat{M}}(D')=Z_{\widehat{M}}(D)$ and the fact that the two extra crossings 
of $D'$ have opposite signs gives $I_{\widehat{M}}(D')=I_{\widehat{M}}(D)$. 	
\end{proof} 

\section{Applications}\label{s:applications}

\subsection{Refined Potts models and the $10_{42}$ diagrams}\label{ss:Potts1042}

Consider a Potts model $M = (W^+_{\rm Potts},d)$ as in Example~$(1)$ from  Subsection~\ref{ss:spinmodels} for $n=3$. Recall that 
\[
\quad W^+_{\rm Potts} = (-\xi^{-3})I + \xi (J-I),
\]
where $d = -\sqrt{3} = -\xi^2-\xi^{-2}$. By Remark~\ref{r:refined-typeII-exist} we have 
$I,J\in N_W$.  
Therefore, any matrix of the form $V^+_{a,b} = aI+b(J-I)$, $a,b\in\bC$ is symmetric and belongs 
to $N_W$. Since $\psi(I)=J$ and therefore $\psi(J) = \psi^2(I) = n I$, we have 
\[
\Psi(V^+_{a,b})= (a+2b) I + (a-b) (J-I) = d V^-.
\]
Hence, if $a(a+2b)\neq 0$ we have a refined spin model of the form 
$\widehat M_{a,b} = (W^+_{\rm Potts},V^+_{a,b},d)$. Let $D_{10_{42}}$ (respectively $D'_{10_{42}}$) 
the central (respectively right-most) symmetric union diagram of Figure~\ref{f:89}. 
A computation with Sage~\cite{T17a} gives 
\[
I_{\widehat M_{a,b}}(D_{10_{42}}) = d\frac{a^3+6 a^2 b + 2 b^3}{a(a +2b)^2}\quad\text{and}\quad
I_{\widehat M_{a,b}}(D'_{10_{42}}) = d\frac{3a}{a+2b}.
\]
Clearly, for infinitely many choices of $(a,b)$ with $a(a+2b)\neq 0$ we have 
\[
I_{\widehat M_{a,b}}(D_{10_{42}})\neq I_{\widehat M_{a,b}}(D'_{10_{42}}),
\] 
and applying Theorem~\ref{t:refined-invariance} we conclude that 
$D_{10_{42}}$ and $D'_{10_{42}}$ are not symmetrically equivalent. This gives 
a partial answer to the question left open by Eisermann and Lamm and described 
at the end of Subsection~\ref{ss:rJp}.~\footnote{All the refined spin models of type II that we were able to use had normalized partition functions which took the same values on $D_{10_{42}}$ and $D'_{10_{42}}$. However, we do not know how relevant this information is for the question whether $D_{10_{42}}$ and $D'_{10_{42}}$ are weakly symmetrically equivalent. In fact, on the one hand, we could not perform a large amount of calculations because their intensity grew very quickly with the size of model. On the other hand, at the time of writing there is no general classification of spin models, therefore some newly discovered spin model could work in the future.}

\subsection{Refined pentagonal models and the $8_9$ diagrams}\label{ss:pentagonal}

Now we consider the pentagonal spin model of Example~(2) from Subsection~\ref{ss:spinmodels}. 
We want to define a refined spin model of the form 
\[
\widehat{M}_{\rm pent} = (W^+_{\rm pent},V^+,d),
\]
where 
\[
W^+_{\rm pent} = I + \om A_1 + \om^4 A_2,\quad
\om = e^{2\pi i/5},\quad
d = 
\sqrt{5}
\]
and 
\[
A_1 = \sm{0&1&0&0&1\\1&0&1&0&0\\0&1&0&1&0\\0&0&1&0&1\\1&0&0&1&0},\quad
A_2 = \sm{0&0&1&1&0\\0&0&0&1&1\\1&0&0&0&1\\1&1&0&0&0\\0&1&1&0&0}.
\]
It is easy to check that both $A_1$ and $A_2$ belong to the Nomura algebra $N_W$. If we 
let $V^+ = a I + b A_1 + c A_2\in N_W$, we need to check for which $a,b,c\in\bC$ 
we have $\al_{V^+}\cdot\al_{V^-}\neq 0$. Clearly $\al_{V^+} = a$ and, since 
$V^+ Y^{W^+}_{aa} = a + 2b + 2c$, we have $\al_{V^-} = (a+2b+2c)/d$. 
Therefore $\widehat{M}_{\rm pent} = (W^+_{\rm pent},V^+,d)$ 
is a refined spin model for every $a,b,c\in\bC$ such that $a(a+2b+2c)\neq 0$. 
Let $D_{8_9}$ be the left-most diagram of Figure~\ref{f:89} and $D'_{8_9}$ the diagram obtained from $D_{8_9}$ by switching all the crossings on the axis. 
A computation with Sage~\cite{T17a} yields 
\[
I_{\widehat{M}_{\rm pent}}(D_{8_9}) = 
d[a(a^2+2ab+2ac+2b^2+2c^2)+(d-1)(b^3+c^3)-(d+1)bc(b+c)]/a^2(a+2b+2c)
\]
and
\begin{align*}
	I_{\widehat{M}_{\rm pent}}(D'_{8_9}) = d [a^2(a+6b+6c) + 2(d+1)a(b^2+c^2)+ 
	(3-d)(b^3+c^3) + 4(1-d)abc +\\ (d-1)bc (b + c)]/a(a+2b+2c)^2.
\end{align*}
In particular, 
\[
I_{\widehat{M}_{\rm pent}}(D_{8_9})_{a=1,\ c=-b} = d(4b^2+1) \neq 
I_{\widehat{M}_{\rm pent}}(D'_{8_9}) _{a=1,\ c=-b} = 40b^2+d
\]
which, by Theorem~\ref{t:refined-invariance}, implies that the diagrams $D_{8_9}$ and $D'_{8_9}$ are not symmetrically equivalent. In fact, if we choose $\xi\in\bC$ such that $\xi^2 = (1-d)/2$ and we set $b=c=\xi\in\bC$ and $a=-\xi^{-3}$, we have $d = - \xi^2 - \xi^{-2}$ and we obtain a Potts-refined spin model (see Remark~\ref{r:refined-typeII-exist}). Substituting these values of $a,b$ and $c$ we get 
\[
I_{\widehat{M}_{\rm pent}}(D_{8_9})_{a=-\xi^{-3},\ b=c=\xi} = -5d+10 \neq 
I_{\widehat{M}_{\rm pent}}(D'_{8_9})_{a=-\xi^{-3},\ b=c=\xi} = -5d-10. 
\]
This shows that $D_{8_9}$ and $D'_{8_9}$ are not weakly symmetrically 
equivalent. 

\subsection{Infinitely many symmetrically inequivalent diagrams}\label{ss:infinitenonequiv}

Eisermann and Lamm~\cite[\S 2.5]{Ei.La07} defined the connected sum between two symmetric union 
diagrams $D$ and $D'$ by putting $D$ above $D'$ along the axis $B$ and then 
symmetrically joining a strand of $D$ transverse to $B$ to a strand 
of $D'$ transverse to $B$. They showed that this results in 
an associative operation which is well--defined on weakly symmetric equivalence classes and denoted 
the connected sum of the symmetric union diagrams $D$ and $D'$ by $D\# D'$. 
Up to applying $S3$ and $S2(h)$ moves, one may always assume that the strands of $D$ of $D'$ 
used for the operation are, respectively, at the very bottom of $D$ and at the very top 
of $D'$ (see~\cite[Fig.~17]{Ei.La07}). Proposition~\ref{p:gluing} below, whose proof will be provided in~Section~\ref{s:gluing}, allows us to establish Theorem~\ref{t:connectedsums} below, which easily implies the existence of infinitely many pairs of symmetrically inequivalent but Reidemeister equivalent symmetric union 
diagrams. 

We need one more definition before we can state Proposition~\ref{p:gluing}. 
View a complex $n\x n$ matrix $A\in \Mat$ as a map $A\co X\x X\to\bC$ with $X=\{1,\ldots, n\}$, and 
let $t\co X\to X$ be the `shift' map given by $t(a) = a+1\mod n$ for each $a\in X$.  
We say that a refined spin model $\widehat{M} = (W^+,V^+,d)$ is {\em translation-invariant} if 
\[
W^\pm(t(a),t(b)) = W^\pm(a,b)\quad\text{and}\quad
V^\pm(t(a),t(b)) = V^\pm(a,b)
\]
for each $a,b\in X$.
\begin{prop}\label{p:gluing}
	Let $\widehat M$ be a translation-invariant, refined spin model and let $D$, $D_1$ and $D_2$ be 
	oriented, symmetric union link diagrams. Suppose that $\Ga_D=\Ga_{D_1}\cup\Ga_{D_2}$, 
	where $\Ga_{D_1}$ and $\Ga_{D_2}$ are subgraphs of $\Ga_D$ intersecting in a single vertex $v_0$. 
	Then, 
	\[
	I_{\widehat{M}}(D) =  \dfrac1d I_{\widehat{M}}(D_1) I_{\widehat{M}}(D_2).
	\]
\end{prop}

\begin{thm}\label{t:connectedsums}
Let $D$, $D'$ be Reidemeister equivalent, oriented symmetric union link diagrams. If 
\[
I_{\widehat{M}}(D) \neq I_{\widehat{M}}(D')
\]
for some translation-invariant refined spin 
model $\widehat{M}$, then for infinitely many $k\geq 1$ the connected sums 
\[
\#^k D = D\#\stackrel{(\text{$k$ times})}{\cdots}\# D \quad\text{and}\quad
\#^k D' = D'\#\stackrel{(\text{$k$ times})}{\cdots}\# D'
\]
are Reidemeister equivalent but not symmetrically equivalent. 
If $\widehat{M}$ is of type II, then $\#^k D$ and $\#^k D'$ are not weakly symmetrically equivalent. 
Moreover, the same conclusions hold for each $k\geq 1$ if either  
$I_{\widehat{M}}(D) = \lambda I_{\widehat{M}}(D')$ or 
$I_{\widehat{M}}(D') = \lambda I_{\widehat{M}}(D)$, 
where $\lambda\in\bR_{\geq 0}$.  
\end{thm}

\begin{proof}
	Proposition~\ref{p:gluing} applies to triples of the form $D = D_1\# D_2$, $D_1$, 
	$D_2$, where the connected sum is performed using a bottom transverse strand of $D_1$ and 
	a top transverse strand of $D_2$, as explained above. Hence, for each $k\geq 1$ we have
	\[
	I_{\widehat{M}}(\#^k D) = \frac1{d^{k-1}}I_{\widehat{M}}(D)^k,\quad\text{and}\quad
	I_{\widehat{M}}(\#^k D') = \frac1{d^{k-1}}I_{\widehat{M}}(D')^k.
	\]
	Therefore, the equality $I_{\widehat{M}}(\#^k D) = I_{\widehat{M}}(\#^k D')$ implies 
	that $I_{\widehat{M}}(D) = \ze I_{\widehat{M}}(D')$, with $\ze^k=1$, and the statement follows easily. 
\end{proof}

\begin{cor}\label{c:connectedsums}
Let $D_{10_{42}}$, $D'_{10_{42}}$, $D_{8_9}$ and $D'_{8_9}$ the symmetric union diagrams 
considered in Subsections~\ref{ss:Potts1042} and~\ref{ss:pentagonal}. Then, for each $k\geq 1$  
the symmetric union diagrams $\#^k D_{10_{42}}$ and $\#^k D'_{10_{42}}$ are Reidemeister equivalent but not symmetrically equivalent, while $\#^k D_{8_9}$ and $\#^k D'_{8_9}$ are Reidemeister but not weakly symmetrically equivalent. 
\end{cor}

\begin{proof}
	Let $\widehat{M}_{a,b}$ be the refined spin model defined in Subsection~\ref{ss:Potts1042}. By the calculations given there we have 
	\[
	I_{\widehat M_{1,0}}(D_{10_{42}}) = -\sqrt{3}\quad\text{and}\quad
	I_{\widehat M_{1,0}}(D'_{10_{42}}) = -3\sqrt{3}.
	\]
	Since the Potts model is translation invariant, applying Theorem~\ref{t:connectedsums} 
	we obtain that, for each $k\geq 1$, the diagrams $\#^k D_{10_{42}}$ and $\#^k D'_{10_{42}}$ are not symmetrically equivalent. Similarly, by the results of Subsection~\ref{ss:pentagonal}, 
	if $\xi^2 = (1-\sqrt{5})/2$ we have 
	\[
	I_{\widehat{M}_{\rm pent}}(D_{8_9})_{a=-\xi^{-3},\ b=c=\xi} = 10 - 5\sqrt{5} \quad\text{and}\quad
	I_{\widehat{M}_{\rm pent}}(D'_{8_9})_{a=-\xi^{-3},\ b=c=\xi} = - 10 - 5\sqrt{5}. 
	\]
	As before, since the pentagonal model is translation invariant we may apply Theorem~\ref{t:connectedsums}. Therefore, for each $k\geq 1$ the diagrams $\#^k D_{8_9}$ and $\#^k D'_{8_9}$ are not weakly symmetrically equivalent. 
\end{proof}
 
\section{Proof of Proposition~\ref{p:gluing}}\label{s:gluing}

Recall from Section~\ref{s:spin-models} that, if $\widehat{M} = (W^+,V^+,d)$ is a refined spin model, 
the normalized partition function of a symmetric union diagram $D$ takes the form 
\[
I_{\widehat{M}}(D) = \al_{V^+}^{-p_B(D)}\al_{V^-}^{-n_B(D)} d^{-N} Z_{\widehat{M}}(D),
\]
where $N=|\Ga^0_D|$ and 
\[
Z_{\widehat{M}}(D) = d^{-N} \sum_{\si\co\Ga^0_D\to X} 
\prod_{e\in\Ga_B^1} V^{s(e)}(\si(v_e),\si(w_e))\prod_{e\in\Ga^1_D\setminus\Ga_B^1} W^{s(e)}(\si(v_e),\si(w_e)).
\]
Here we are omitting the coloring from the notation because of Corollary~\ref{c:color-indep}.
Fix a vertex $v_0\in\Ga_D^0$ and an element $a\in X$. Define
\begin{equation}\label{e:}
R_{\widehat{M}}(D,v_0;a) := \sum_{\si\ |\ \si(v_0) = a} 
\prod_{e\in\Ga_B^1} V^{s(e)}(\si(v_e),\si(w_e))\prod_{e\in\Ga^1_D\setminus\Ga_B^1} W^{s(e)}(\si(v_e),\si(w_e)).
\end{equation}
Then, we have 
\[
Z_{\widehat{M}}(D) = d^{-N}\sum_{a\in X} R_{\widehat{M}}(D,v_0;a).
\]

\begin{lemma}\label{l:rel-abs-inv} 
	If $\widehat{M} = (W^+,V^+,d)$ is a translation-invariant refined spin model,  
	\[
	Z_{\widehat{M}}(D) = n d^{-N} R_{\widehat{M}}(D,v_0; a)
	\] 
	for each $v_0\in\Ga^0_D$ and $a\in X$.
\end{lemma}

\begin{proof} 
	It suffices to show that $R_{\widehat{M}}(D,v_0;a) = R_{\widehat{M}}(D,v_0; t(a))$ for 
	each $a\in X$, where $t(a) = a+1\mod n$. 
	Let 
	\[
	w(\si,e) = 
	\begin{cases} 
	W^{s(e)}(\si(v_e),\si(w_e))\ \text{if $e$ is off the axis}\\
	V^{s(e)}(\si(v_e),\si(w_e))\ \text{if $e$ is on the axis},
	\end{cases} 
	\]
	where $s(e)\in\{+,-\}$ is the sign of $e$. Since the spin model is translation-invariant, 
	\[
	\sum_{\si\ |\ \si(v_0) = a} \prod_e w (\si, e) = 
	\sum_{\si\ |\ t\circ\si(v_0) = t(a)} \prod_e w (t\circ\si, e) = 
	\sum_{\si\ |\ \si(v_0) = t(a)} \prod_e w (\si, e),
	\]
	which implies the required identity. 
\end{proof}

	Clearly 
	%
	\[
	R_{\widehat{M}}(D,v_0;a) = R_{\widehat{M}}(D_1,v_0;a) R_{\widehat{M}}(D_2,v_0;a)
	\] 
	for each $a\in X$. Let $N=|\Ga_D^0|$, $N_1= |\Ga_{D_1}^0|$ and 
	$N_2=|\Ga_{D_2}^0|$. Then, we have $N=N_1 + N_2 - 1$. Now choose any $x_0\in X$. Since $d^2=n$, $p_{B}(D) = p_{B}(D_1) + p_{B}(D_2)$ and $n_{B}(D) = n_{B}(D_1) + n_{B}(D_2)$, in view of Lemma~\ref{l:rel-abs-inv} we have
	\begin{align*}
	d I_{\widehat{M}}(D)  & =   \al_{V^+}^{-p_B(D)}\al_{V^-}^{-n_B(D)} d^{-N+1} Z_{\widehat{M}}(D) \\
	& = n_B l_{V^+}^{-p_B(D)}\al_{V^-}^{-n_B(D)} d^{-N+1} R_{\widehat{M}}(D,v_0;x_0) \\ 
	& = d^{-N_1} \al_{V^+}^{-p_B(D_1)}\al_{V^-}^{-n_B(D_1)}  n R_{\widehat{M}}(D_1,v_0;x_0)\  
	d^{-N_2} \al_{V^+}^{-p_B(D_2)}\al_{V^-}^{-n_B(D_2)}  n R_{\widehat{M}}(D_2,v_0;x_0) \\
	& = I_{\widehat{M}}(D_1) I_{\widehat{M}}(D_2).
	\end{align*}
This concludes the proof of Proposition~4.1.

\bibliographystyle{abbrv}
\bibliography{biblio}

\end{document}